\documentclass[11pt, reqno]{amsart}
\usepackage{amsfonts}
\usepackage{amssymb}
\usepackage{latexsym}
\usepackage[]{enumerate}
\usepackage{verbatim}
\usepackage[]{enumerate}

\usepackage[final]{hyperref}
\hypersetup{colorlinks=true, linkcolor=blue, anchorcolor=blue, citecolor=red, filecolor=blue, menucolor=blue, pagecolor=blue, urlcolor=blue}

\newcommand{\abs}[1]{\vert #1 \vert}

\newcommand{\Bigabs}[1]{\Bigl\vert #1 \Bigr\vert}

\newcommand{\norm}[1]{\left\Vert #1 \right\Vert}

\newcommand{\Z}{\mathbb{Z}}
\newcommand{\R}{\mathbb{R}}

\newcommand{\angles}[1]{\langle #1 \rangle}

\DeclareMathOperator{\sgn}{sgn}
\DeclareMathOperator{\supp}{supp}

\newtheorem{theorem}{Theorem}

\newtheorem{lemma}{Lemma}
\newtheorem{corollary}{Corollary}

\theoremstyle{definition}

\theoremstyle{remark}
\newtheorem{remark}{Remark}

\title
[
	Well-posedness of Whitham--Boussinesq type system
]
{
	Well-posedness for a dispersive system
	of the Whitham--Boussinesq type
}

\author{ E. Dinvay, S. Selberg and A. Tesfahun}

\email{Evgueni.Dinvay@uib.no, Sigmund.Selberg@uib.no, achenef@gmail.com}
 
\address{Department of Mathematics\\
University of Bergen\\
PO Box 7803\\
5020 Bergen\\ Norway}

\subjclass[2010]{35Q55}

\begin{document}

\begin{abstract} 
We regard the Cauchy problem for
a particular Whitham--Boussinesq system
modelling surface waves
of an inviscid incompressible fluid layer.
We are interested in well-posedness
at a very low level of regularity.
We derive dispersive and Strichartz estimates,
and implement them together
with a fixed point argument to solve the problem
locally.
Hamiltonian conservation guarantees global
well-posedness for small initial data
in the one dimensional settings.
\end{abstract}

\maketitle

\section{Introduction}
\setcounter{equation}{0}

We consider the following Whitham-type system posed on ${\R^{1+1}}$
\begin{equation}
\label{wt1}
\left\{
\begin{aligned}
  \partial_t \eta +  \partial_x v & = -K_1^2 \partial_x (\eta v)
  \\
  \partial_t v +  K_1^2\partial_x \eta  &=   - K_1^2 \partial_x( v ^2/2),
\end{aligned}
\right.
\end{equation} 
where 
\begin{equation}
\label{K_definition}
	K_1:=K_1(D)=\sqrt{ \tanh(D)/ D} \quad \text{with}  \ D=-i\partial_x
	.
\end{equation}
The operator $K_1$ is a Fourier multiplier operator 
with the symbol $\xi \mapsto \sqrt{ \tanh \xi / \xi}$.
It is bounded and invertible in
$L^2(\mathbb R)$,
more precisely,
it is a linear isomorphism from
$L^2(\mathbb R)$ to $H^{1/2} (\mathbb R)$.
Its inverse $K_1^{-1}$ is equivalent to the Bessel potential $J^{1/2}$
defined by the symbol
$\xi \mapsto \left( 1 + \xi^2 \right) ^{1 / 4}$.
Functions $\eta$, $v$ are assumed to be real valued.
Note that
$K_1^2 \partial_x = i \tanh D$
and so System \eqref{wt1} has a semilinear nature.

We complement \eqref{wt1} with the initial data 
\begin{equation}
\label{data1}
\eta(0)= \eta_0 \in H^s(\R)
, \qquad
v(0)= v_0 \in H^{s+1/2}(\R)
,
\end{equation} 
where
$H^s = \left(1-\partial_x^2 \right)^{-s/2} L^2(\R)$ is the standard notation for the Sobolev space of order $s$.
Such initial value problem
describes evolution with time of surface waves of
a liquid layer.
The model approximates the two-dimensional water wave problem
for an inviscid incompressible potential flow.
The variables $\eta$ and $v$ denote the surface elevation
and fluid velocity, respectively.
For some discussion on its precise physical meaning
we refer the reader
to the work by Dinvay, Dutykh and Kalisch
\cite{Dinvay_Dutykh_Kalisch},
where the system \eqref{wt1} appeared for the first time.
Formally, $v$ equals $i \tanh D$-derivative of
the velocity potential trace on surface
associated with the irrotational velocity field.
In the long wave Boussinesq regime
$v$ coincides with the horizontal
fluid velocity at the surface.

The system \eqref{wt1} possesses a Hamiltonian structure
\cite{Dinvay_Dutykh_Kalisch}.
To our knowledge, there are at least two conserved quantities
associated with this system.
The first one
\begin{equation}
\label{Hamiltonian1}
	\mathcal H(\eta, v)  = \frac 12 \int_\R
	\left(	
		\eta^2 + v K_1^{-2} v
		+ \eta v^2
	\right)
	dx
\end{equation}
has the meaning of total energy.
The second one
\begin{equation*}
	\mathcal I(\eta, v)  = \int_\R
	\eta K_1^{-2} v dx
\end{equation*}
has the meaning of momentum.
The system \eqref{wt1}
has a Hamiltonian structure
of the form
\[
	\partial_t (\eta, v)^T = \mathcal J \nabla \mathcal H(\eta, v)
\]
with the skew-adjoint matrix
\[
	\mathcal J
	=
	\begin{pmatrix}
		0 & - i \tanh D
		\\
		- i \tanh D & 0
	\end{pmatrix}
	,
\]
which in particular
guarantees conservation of the energy functional $\mathcal H$. It is worth to notice that System \eqref{wt1} can be derived
at least formally
in the long wave asymptotic regime
from the Zakharov-Craig-Sulem formulation
of the water wave problem
\cite{Lannes} also known to be Hamiltonian.
The Hamiltonian structure of the Zakharov-Craig-Sulem formulation
is canonical,
in the sense that the corresponding skew-adjoint matrix
\(
	\mathcal J
	=
	\begin{pmatrix}
		0 & 1
		\\
		-1 & 0
	\end{pmatrix}
	.
\)
It is interesting to notice that
Model \eqref{wt1} also enjoys a canonical
Hamiltonian structure,
which is directly comparable with the one
of the full water wave system,
when using variables $(\eta, \psi)$
where $\psi$ is such that
$v = i \tanh D \psi$.
Numerical simulations done in \cite{Dinvay_Dutykh_Kalisch}
show how insignificantly values of functional
$\mathcal H$ differ from the corresponding
energy levels of the full water problem.

We also consider a system posed on ${\R^{2+1}}$ of
the following Whitham-Boussinesq type
\begin{equation}
\label{wt2d}
\left\{
\begin{aligned}
  	\partial_t \eta + \nabla \cdot \mathbf v
  	& =
	- K_2^2 \nabla \cdot (\eta  \mathbf v)
	,
  	\\
  	\partial_t  \mathbf v + K_2^2 \nabla \eta
 	&=
 	- K_2^2 \nabla \left( | \mathbf v|^2/2 \right),
\end{aligned}
\right.
\end{equation} 
where $ \mathbf v= (v_1, v_2)\in \R^2$ is a curl free vector field, i.e., $
 	\nabla \times  \mathbf v   = 0$, and 
$$
	K_2:=K_2(D)=\sqrt{ \tanh|D| / |D|} \quad ( D=-i \nabla)
$$
with the corresponding symbol
$
	K_2(\xi)
	=
	\sqrt{
		\tanh (|\xi|) / |\xi|
	}
	.
$
We complement \eqref{wt2d} with the initial data 
\begin{equation}
\label{data2d}
\eta(0)= \eta_0 \in H^s \left( \R^2 \right)
, \qquad
\mathbf v(0)= \mathbf v_0
\in \left[ H^{s+1/2} \left( \R^2 \right) \right] ^2
.
\end{equation} 
This is a two dimensional analogue of System \eqref{wt1}
describing evolution with time of surface waves of
a liquid layer in the three dimensional physical space.
As above the variables $\eta$ and $\mathbf v$
denote the surface elevation
and the fluid velocity, respectively.
The system enjoys the Hamiltonian structure
\[
	\partial_t (\eta,  \mathbf  v)^T = \mathcal J \nabla \mathcal H(\eta ,  \mathbf v)
\]
with the skew-adjoint matrix
\[
	\mathcal J
	=
	\begin{pmatrix}
		0 & - K_2^2 \partial_{x_1} & - K_2^2 \partial_{x_2}
		\\
		- K_2^2 \partial_{x_1} & 0 & 0
		\\
		- K_2^2 \partial_{x_2} & 0 & 0
	\end{pmatrix}
	,
\]
which in particular
guarantees conservation of the energy functional
\begin{equation}
\label{Hamiltonian2}
	\mathcal H(\eta, \mathbf v)  = \frac 12 \int_{\R^2}
	\left(	
		\eta^2 + \left| K_2^{-1} \mathbf v \right|^2
		+ \eta |\mathbf v |^2
	\right)
	dx
	.
\end{equation}

Equations \eqref{wt1} were firstly proposed and studied numerically
in \cite{Dinvay_Dutykh_Kalisch}.
Later in \cite{Dinvay} the first proof of local
well-posedness
based on an energy method
and a compactness argument was given.
System \eqref{wt1} is an alternative to other
weakly nonlinear dispersive models
describing two-wave propagation \cite{Dinvay_Dutykh_Kalisch}.
Those models  are in a good agreement with
experiments \cite{Carter}.
They also have
many peculiarities of
the full water wave problem.
The existing results on well-posedness theory,
however, are not completely satisfactory.
To our knowledge, apart from the model under consideration,
there is only one local well-posedness result so far for the regarded system in \cite{Dinvay_Dutykh_Kalisch} that have been proved
by Pei and Wang \cite{Pei_Wang}.
To achieve this the authors
imposed an additional non-physical condition
$\eta \geqslant C > 0$.
The initial value problem regarded in \cite{Pei_Wang}
is probably ill-posed for large data
if one removes the positivity  assumption $\eta > 0$, as
an heuristic argument given in \cite{Klein_Linares_Pilod} shows.
Recently, Kalisch and Pilod \cite{Kalisch_Pilod}
have proved local well posedness
for a surface tension regularisation of the system
from \cite{Pei_Wang}.
They were able to exclude the positivity assumption $\eta > 0$.
However, the maximal time of existence for their
regularisation is bounded by the capillary parameter.
One does not need any regularisation or special
non-physical conditions to claim the well posedness
for \eqref{wt1}, \eqref{data1}.

In fact \eqref{wt1} can be regarded itself
as a regularization of the
system introduced by Hur and Pandey \cite{Hur_Pandey}.
The latter was also investigated numerically
in \cite{Dinvay_Dutykh_Kalisch} and compared
with other models of Whitham-Boussinesq type.
Admitting formally
$\tanh D \thicksim D$ for small frequencies and
substituting $D$ instead of $\tanh D$ to the nonlinear
part of Equations \eqref{wt1}, one comes
to the system regarded in \cite{Hur_Pandey}.
Hur and Pandey have proved
the Benjamin--Feir instability \cite{Hur_Pandey}
of periodic travelling waves for their system,
which makes it valuable.
If one in addition formally discards the term $\eta \partial_x u$
in the system given in \cite{Hur_Pandey},
then a new alternative system turns out to be locally well-posed
and features wave breaking \cite{Hur_Tao}.
However, the latter does not belong to
the class of Boussinesq--Whitham models
since nonlinear non-dispersive terms have been neglected.

We would like to pay special attention to
a system that was not considered in \cite{Dinvay_Dutykh_Kalisch}
but was introduced by
Duch\^ene, Israwi and Talhouk \cite{Duchene_Israwi}.
They modified the bi-layer Green-Naghdi model
improving the frequency dispersion.
In fact, their system is also linearly fully dispersive,
which makes it a close relative to System \eqref{wt1}.
Note that their system is Hamiltonian as well.
Moreover, they have justified the Green-Naghdi modification
proving well-posedness, consistency and convergence
to the full water wave problem in the
Boussinesq regime \cite{Duchene_Israwi}.
In addition, consistency of Hamiltonian structure
is shown, so that energy levels of
the approximate model can be compared with
the full water energy.
Existence of solitary waves for their system
is also proved in \cite{Duchene_Nilsson_Wahlen}.
Returning to the system regarded
by Pei and Wang \cite{Pei_Wang},
we should notice that a question of
existence of solitary waves for it,
is closed as well \cite{Nilsson_Wang}.
Finally, we point out that well-posedness of
the modified Green-Naghdi model is satisfactory,
in the sense
that it needs neither surface tension nor
any non-physical initial condition.
All this together makes it a promising system.
And indeed, as noticed in \cite{Duchene_Israwi},
their modification gives more reliable
results when it comes to
large-frequency Kelvin-Helmholtz instabilities
than other models of the Green-Naghdi type.

On the contrary,
System \eqref{wt1} has a couple of advantages
compared with the modified Green-Naghdi
model \cite{Duchene_Israwi}.
Firstly,
it is derived, though not rigorously,
from the Zakharov-Craig-Sulem formulation,
and as a result
one knows the relation between variables
$(\eta, v)$ and those describing
the full potential fluid flow \cite{Dinvay_Dutykh_Kalisch}.
As to the modification discussed, it is presented
in variables where the first one has the meaning
of the surface elevation and so coincides with $\eta$.
Its dual variable is called the layer-averaged horizontal velocity
\cite{Duchene_Israwi}.
In the Boussinesq regime it definitely coincides with
the same object associated with the full Euler equations.
However, one cannot guarantee that it will be the case
in shorter wave regimes.
Whereas for Whitham type models one might anticipate
a good agreement which is confirmed by experiments
\cite{Carter}.
Here we must admit that neither
the Whitham-Boussinesq system \eqref{wt1}
nor the modified Green-Naghdi system
are tested by Carter \cite{Carter}.
So it might be only a matter of time
before the modified Green-Naghdi velocity is
given an exact physical meaning.
In other words, we expect that this velocity
will be associated with the full water problem notions.
The second issue is that it does not seem obvious
how the modified Green-Naghdi system can be generalized 
to a three-dimensional model,
whereas for System \eqref{wt1} it is straightforward.

Let us formulate the main results.
The first one is an improvement of the
local existence claimed in \cite{Dinvay}.

\begin{theorem}
[Local existence in 1d]
\label{simple_theorem}
	Let $s > -1/10$.
	Given any $R > 0$ there exists a time $T = T(R) > 0$ such that for any initial data
	$(\eta_0, v_0) \in X^s := H^s(\mathbb{R}) \times H^{s + 1/2}(\mathbb{R})$
	with norm
	\(
		\lVert \eta_0 \rVert _{H^s}
		+ \lVert v_0 \rVert _{H^{s + 1/2}} \le R,
	\)
	there exists a solution $( \eta, v )$ in the space
	\(
		 X^s_T := C([0, T];
		H^s(\mathbb{R}) \times H^{s + 1/2}(\mathbb{R}) )
	\)
	of the Cauchy problem \eqref{wt1}, \eqref{data1}.
	Moreover, the solution is unique in a subspace of $X^s_T$ and it depends continuously on the initial data.
\end{theorem}

\begin{theorem}
[Local existence in 2d]
\label{simple_theorem2d}
	Let $s > 1/4$.
	Given any $R > 0$ there exists a time $T = T(R) > 0$ such that for any initial data
	\(
		(\eta_0, \mathbf v_0) \in X^s := H^s \left( \mathbb R^2 \right)
		\times
		\left( H^{s + 1/2} \left( \mathbb R^2 \right) \right) ^2
	\)
	with $\nabla \times \mathbf v_0 = 0$ and with norm
	\(
		\lVert \eta_0 \rVert _{H^s}
		+ \lVert \mathbf v_0 \rVert _{ (H^{s + 1/2})^2 } \le R,
	\)
	there exists a solution $( \eta, \mathbf v )$ in the space
	\(
		X^s_T := C \left([0, T];
		H^s(\mathbb{R}^2)
		\times
		\left( H^{s + 1/2} \left( \mathbb R^2 \right) \right) ^2
		\right)
	\)
	of the Cauchy problem \eqref{wt2d}, \eqref{data2d}.
	Moreover, the solution is unique in a subspace of $X^s_T$ and it depends continuously on the initial data.
\end{theorem}

\begin{remark}
	For $s > 0$ in 1d and $s > 1/2$ in 2d
	the solution is unique in the whole space $X^s_T$.
	Moreover, the flow map is real analytic for such values of $s$.
\end{remark}

Theorem \ref{simple_theorem}
does not rely on the non-cavitation hypothesis
$1 + \eta > 0$, since smallness of waves is implied in the model.
It can be seen as a drawback comparing with
the model from \cite{Duchene_Israwi}.
However, as mentioned above, it is difficult
to say for now which one of these two competing models
is a better approximation to the Euler equations.
Instead of the non-cavitation, there is another condition
that we have to impose to prove the following global result.
The meaning of this new condition is that the total
energy should be positive and not too big.
We point out that this condition is imposed
at the energy level of regularity
and is independent on the regularity $s$ of the initial data.

\begin{theorem}
[Global existence in 1d]
\label{mainthm}
	Assume that $s\geqslant 0$ and consider the local solution from Theorem \ref{simple_theorem}.
	There exists $\delta>0$ such that if
	$$
		\| \eta_0\|_{L^2(\R)}+ \| v_0\|_{H^{1/2}(\R)}
		\leqslant \delta
	$$
	then the solution extends to a global-in-time solution 
	$$
		(\eta, v) \in C\left( \R; H^s(\R) \times H^{s+1/2}(\R)\right).
	$$
\end{theorem}

In the sections below, we first diagonalize
Systems \eqref{wt1} and  \eqref{wt2d}
and reformulate the local theorems in the new variables.
Then we demonstrate how the local result can be obtained
in less general settings applying an elegant classical
PDE technique based on the standard Sobolev embedding.
This also demonstrates the necessity of dispersive estimates
for going down to the energy level of regularity $s = 0$ in 1d.
Note that the domain of the Hamiltonian functional
\eqref{Hamiltonian1} is
\(
	L^2(\R) \times H^{1/2}(\R)
	.
\)
After that we obtain estimates of Strichartz type
studying asymptotic behaviour of a particular
oscillatory integral
(see Lemma \ref{lm-dispest} and its proof below). 
This is an improvement comparing with dispersive
estimates obtained in \cite{Benoit}.
In fact we have $L^{\infty}$-norm decay
dominated by $L^1$-norm locally in frequency,
which gives us localised Strichartz estimates.
Whereas the decay in \cite{Benoit} is
dominated by weighted Sobolev spaces,
though frequency independent.
With the new estimates in hand we can apply
the fixed point argument in a ball
of the Bourgain space associated with
the water wave dispersion.
This gives us the local existence theorems, Theorems
\ref{simple_theorem} and \ref{simple_theorem2d}.

The last step is to prove the global well-posedness
theorem \ref{mainthm}. 
For $s = 0$ it comes straightforwardly from the
energy \eqref{Hamiltonian1} conservation
via the continuity argument and the local result.
For $s > 0$ we prove the persistence of regularity.
Surprisingly, it is not enough just to have the dispersive
Strichartz estimates to claim the persistence.
Thankfully, our velocity variable $v$ is bounded in
$H^{1/2}$-norm and so we are able to use
the following limiting case of the Sobolev embedding theorem.

\begin{lemma}
[Brezis-Gallouet inequality]
\label{Brezis_lemma}
	Suppose $f \in H^s(\mathbb R^d)$ with $s > d / 2$.
	Then
	\begin{equation}
	\label{Brezis_inequality}
		\lVert f \rVert_{L^{\infty}}
		\leqslant
		C_{s, d}
		\left(
			1 + \lVert f \rVert_{H^{d/2}}
			\sqrt{ \log( 2 + \lVert f \rVert_{H^s} ) }
		\right)
		.
	\end{equation}
\end{lemma}

Inequality \eqref{Brezis_inequality} was firstly put forward
and proved for a domain in $\mathbb R^d$ with $d= 2$
in the work by Brezis, Gallouet
\cite{Brezis_Gallouet}.
It was extended to the other Sobolev spaces in
\cite{Brezis_Wainger}.
An implementation of this inequality for deriving
a global a priori estimate can be found, for example,
in the work by Ponce \cite{Ponce}
on the global well-posedness of the Benjamin-Ono equation.
We apply a similar trick here, and so that we
repeat the formulation of Lemma \ref{Brezis_lemma}
as it is given in \cite{Ponce}.
This provides us with the persistence of regularity
that in turn concludes the proof of Theorem \ref{mainthm}.

Let us finally give some explanations for the choice of strategy,
focusing on the one dimensional case.
The local well-posedness for $s > 0$ follows from the
standard technique related to semilinear equations.
It requires only Duhamel's formula and suitable
product estimates for the right hand side of \eqref{wt1}
in the Sobolev-based space $X^s = H^s \times H^{s + 1/2}$.
Global bound in $X^0$ follows from the Hamiltonian conservation,
since
\(
	\mathcal H (\eta, v) \approx \norm{(\eta, v)}_{X^0}^2
\)
provided $\norm{(\eta, v)}_{X^0}$ is small.
Hence the global well-posedness in $X^s$ with $s > 0$
follows from the local result and
an a priori bound obtained from
the persistence of regularity and the Brezis-Gallouet inequality.

The main focus of the work is on lowering the regularity
threshold  for the local well-posedness
through the use of dispersive estimates.
One anticipates that even the weak dispersive
properties of System \eqref{wt1}
can lower the threshold at least to the limit case $s=0$.
This together with the global bound automatically
gives us the global well-posedness in $X^0$.
However, the weakness of dispersion means
that the time-decaying $L^1 \to L^{\infty}$-boundedness
of the semigroup, associated with the linearised system,
does not hold.
As a result the standard strategy based on
Strichartz estimates is unavailable.
So instead, we obtain the decay estimate on each component
of the dyadic Littlewood-Paley decomposition with a sharp dependence
on the dyadic number.
From this local decay we deduce
bilinear estimates in the Bourgain space associated with
the water wave dispersion relation.
The local well-posedness is deduced from Duhamel's formula
with the help of these bilinear estimates.

The main peculiarity of the two dimensional case
is that with this technique
we are able to prove the local well-posedness
in $X^s = H^s \times H^{s + 1/2} \times H^{s + 1/2}$ only for 
$s > 1/4$.
It still leaves a gap from the energy space $X^0$,
too big to claim global existence.
Moreover, even in 1d it is not clear so far
if the problem is globally well-posed for some
$s \in (-1/10, 0)$.

Another interesting thing one can notice is that
in the two dimensional case
we were able to get the maximal gain of $d/8$ derivatives
with respect to the naive estimate based only on the unitary
property of the semigroup.
This is optimal in view of the known smoothing of
$\exp \left( it|D|^{1/2} \right)$
that is essentially the semigroup under consideration.
We refer to \cite{Ai2019, Alazard_Burq_Zuily2018}
for more details.
It is interesting to notice that in the one dimensional case
we obtained the gain of $1/10$ derivatives
that turns out to be the same for the full water wave problem
\cite{Ai2019}.
The question remains open if one can improve the result
and lower the threshold
from $s > -1/10$ to the optimal $s > -1/8$ in 1d.

\section{
	Diagonalization of \eqref{wt1} and \eqref{wt2d}, and 
	reformulations of the local existence theorems
}
\setcounter{equation}{0}

We diagonalize \eqref{wt1} as follows.
Defining the new variables
\begin{align*}
u^+_1=\frac{K_1\eta +v}{2 K_1},
\qquad 
 u^-_1=\frac{K_1\eta -v}{2 K_1}
\end{align*}
we have
\begin{equation}\label{upm}
\eta=u^+_1+ u^-_1, \quad v=K_1 (u^+_1 - u^-_1).
\end{equation}
Then we can write the equation for $u^\pm_1$ as follows: 
\begin{align*}
2 K_1 \partial_t u^\pm_1
&=K_1\eta_t \pm  v_t
\\
&= -K_1\partial_x v-K_1^3 \partial_x(\eta v) \mp K_1^2\partial_x \eta\mp  K_1^2\partial_x (v^2/2)
\\
&=  \mp i DK_1 (K_1  \eta \pm v)-iDK_1^2 [ K_1(\eta v)\pm v^2/2].
\end{align*}
Thus, 
\begin{equation}\label{upmeq}
i \partial_t u^\pm_1
= \pm  DK_1 u^\pm_1 +\frac{ DK_1}2[ K_1(\eta v)\pm v^2/2].
\end{equation}

The nonlinear terms can also be written in terms of $u^\pm_1$ as
\begin{equation}
\label{nleq}
 \eta v = (u^+_1+u^-_1) K_1 (u^+_1 - u^-_1),
 \qquad 
 v^2= [K_1(u^+_1 - u^-_1)]^2.
\end{equation}
Now let
 $$m_1(D)=DK_1(D).$$
From \eqref{upmeq}--\eqref{nleq} we see that the system \eqref{wt1} transforms to 
\begin{equation}
\label{wt2}
\left\{
\begin{aligned}
 (i\partial_t-m_1(D) ) u^+_1
&=  B^+_1(u^+_1, u^-_1) ,
\\
( i\partial_t +m_1(D)) u^-_1
&= B^-_1(u^+_1, u^-_1) ,
\end{aligned}
\right.
\end{equation}
where
\begin{equation}
\label{wtnonlin}
4B^\pm_1(u^+_1, u^-_1)  =   D K_1\left[2K_1\left\{ (u^+_1+u^-_1) K_1 (u^+_1 - u^-_1)\right\} \pm  [K_1(u^+_1 - u^-_1)]^2\right],
\end{equation}
The initial data \eqref{data1} transforms to 
\begin{equation}
\label{data2}
	u^\pm_1 (0)= f^\pm_1 :=\frac{K_1\eta_0 \pm v_0}{2 K_1} \in H^s(\R),
\end{equation}
where we used the fact that $K_1(\xi)\sim \angles{\xi}^{-1/2}$, and hence 
$$
	\|K_1^{-1}v_0\|_{H^s(\R)}
	\sim
	\| \angles{D}^{1/2}v_0\|_{H^s(\R)}
	=
	\| v_0\|_{H^{s + 1/2}(\R)}.
$$
Here and below we use the notation
$\angles{\xi} = \sqrt{1 + \xi^2}$,
so $\angles{D} = J$ is the Bessel potential of order $-1$.

To diagonalize \eqref{wt2d} we define
\begin{align*}
u^\pm_2=\frac{K_2|D|\eta \mp i\nabla \cdot \mathbf v}{2 K_2|D|}
\end{align*}
Hence 
\begin{equation}\label{upm2d}
\eta=u^+_2+u^-_2 \quad \mathbf v=-i |D|^{-1} K_2\nabla(u^+_2- u^-_2),
\end{equation}
where we used the fact that $\mathbf v $ is curl free which in turn implies 
$\nabla \nabla \cdot \mathbf v= \Delta  \mathbf v=- |D|^2  \mathbf v$.
Then the equations for $u^\pm_2$ are written as follows: 
\begin{align*}
2 K_2|D| \partial_t u^\pm_2
&=K_2|D|\eta_t \mp i \nabla \cdot \mathbf v_t
\\
&=  \mp iK_2 |D| (K_2|D|\eta \mp i\nabla \cdot \mathbf v)+i|D|^2K_2^2 [  K_2  |D|^{-1}(i\nabla) \cdot(\eta \mathbf v)\pm  (|\mathbf v|^2)/2].
\end{align*}
Thus, 
\begin{equation}\label{u+-eq2d}
i \partial_t u^\pm_2
=  \pm  |D| K_2u^\pm_2 -\frac{ |D|  K_2}2 [ iK_2  R \cdot(\eta \mathbf v)\mp |\mathbf v|^2/2)],
\end{equation}
where $R=(R_1, R_2)$ with $R_j= \partial_j/|D|$ being the Riesz transforms.
Now setting
 $$ m_2(D):=|D|K_2(D)$$ and combining the equations
 \eqref{upm2d}--\eqref{u+-eq2d}  we see that the system \eqref{wt2d-1} transforms to 
\begin{equation}
\label{wt2d-1}
\left\{
\begin{aligned}
 (i\partial_t-m_2(D) ) u^+_2
&=  B^+_2(u^+_2, u^-_2) ,
\\
( i\partial_t +m_2(D)) u^-_2
&= B^-_2(u^+_2, u^-_2) ,
\end{aligned}
\right.
\end{equation}
where
\begin{equation}
\label{wtnonlin-2d-1}
4B^\pm_2(u^+_2, u^-_2)  = - |D|K_2\left[2K_2    R \left\{ (u^+_2+u^-_2)  K_2R (u^+_2 - u^-_2)\right\} \mp  \Bigabs{ K_2R (u^+_2 - u^-_2)}^2\right].
\end{equation}

The initial data \eqref{data2d} transforms to 
\begin{equation}
\label{data2d-1}
	u^\pm_2 (0)= f^\pm_2 :=\frac{K_2|D|\eta_0 \mp i\nabla \cdot \mathbf v_0}{2 K_2|D|} \in H^s(\R),
\end{equation}
where we used the fact that $K_2(\xi)\sim \angles{\xi}^{-1/2}$.

Now let us reformulate Theorem \ref{simple_theorem} and Theorem \ref{simple_theorem2d} 
in terms of the new variables as follows.
\begin{theorem}
\label{lwpthm}
	Let $s> -1/10$.
	Given any $R > 0$ there exists a time $T = T(R) > 0$ such that for any initial data $(f^+_1,f^-_1) \in H^s(\R) \times H^s(\R)$ with norm $\| f^+_1\|_{H^s(\R)} + \| f^-_1\|_{H^s(\R)} 
	\le R$, the Cauchy problem \eqref{wt2}--\eqref{data2} has a solution
	\[
		(u^+_1,u^-_1) \in C \left( [0, T]; H^s(\R) \times H^s(\R) \right).
	\]
	Moreover, the solution is unique in a subset of this space and depends continuously on the data.
\end{theorem}

\begin{theorem}
\label{lwpthm2d-1}
	Let $s> 1/4$.
	Given any $R > 0$ there exists a time $T = T(R) > 0$ such that for any initial data $(f^+_2,f^-_2) \in H^s(\R^2) \times H^s(\R^2)$ with norm $\| f^+_2\|_{H^s(\R^2)} + \| f^-_2\|_{H^s(\R^2)} 
	\le R$, the Cauchy problem \eqref{wt2d-1}--\eqref{data2d-1} has a solution
	\[
		(u^+_2,u^-_2) \in C \left( [0, T]; H^s(\R^2) \times H^s(\R^2) \right).
	\]
	Moreover, the solution is unique in a subset of this space and depends continuously on the data.
\end{theorem}


The system \eqref{wt2}--\eqref{data2} 
can be written in the form of integral equations as 
\begin{equation}
\label{inteqwt1}
u^\pm_1(t) = e^{\mp it m_1(D)}   f^\pm_1 \mp i  \int_0^t  e^{\mp i(t-s) m_1(D)}  B^\pm_1(u^+_1, u^-_1)(s) \, ds.
\end{equation}
Similarly, 
the system  \eqref{wt2d-1}--\eqref{data2d-1}
can be written in the form of integral equations as 
  \begin{equation}
\label{inteqwt2d-1}
u^\pm_2(t) = e^{\mp it m_2(D)}  f^\pm_2\mp i  \int_0^t  e^{\mp i(t-s) m_2(D)}  B^\pm_2(u^+_2, u^-_2)(s) \, ds.
\end{equation}

Applying the contraction argument to \eqref{inteqwt1} together with
the Sobolev embedding one can prove Theorem
\ref{lwpthm} for $s > 0$ and Theorem \ref{lwpthm2d-1} for $s>1/2$, as shown in the next section.
However, to prove Theorem
\ref{lwpthm} for $s>-1/10$ and Theorem \ref{lwpthm2d-1} for $s>1/4$ we need to derive dispersive estimates
on the semigroups $S_{m_d}(\pm t):=e^{ \mp it m_d(D)} $, where
\begin{align*}
m_1(\xi)&=\xi  K_1(\xi)
	=\xi
	\sqrt{ \frac{
		\tanh \xi }{ \xi}}  \qquad (\xi\in \R) ,\\
	m_2(\xi)&=|\xi | K_2(\xi)
	=|\xi|
	\sqrt{ \frac{
		\tanh |\xi |}{ |\xi|}} \qquad (\xi\in \R^2).
		\end{align*}

\section{Non-dispersive estimates}
\setcounter{equation}{0}
\subsection
{
	Local well-posedness for $s>0$ in 1d
}

In this section we prove the local
well-posedness in $H^s \times H^{s + 1/2}$ with $s > 0$ for
System \eqref{wt1} applying a fixed-point argument.
It is only a particular case of Theorem \ref{simple_theorem}
(or of the equivalent theorem \ref{lwpthm}).
In this sense, the section has mainly
an illustrative character.
However, the proof is elegant and does not need any use of
dispersive techniques.
The idea is close to the one used in \cite{Bona_Tzvetkov},
for instance.
This allows us to think about System \eqref{wt1}
as a fully dispersive bi-directional relative to
the BBM equation.

Regard the Whitham operator $K = \sqrt{\tanh D / D}$
and introduce the space $X^s = H^s \times H^{s + 1/2}$
equipped with the norm
\begin{equation}
\label{Xs_norm_definition}
	\lVert (f, g) \rVert_{X^s}^2 =
	\lVert f \rVert_{H^s}^2 + \lVert K^{-1}g \rVert_{H^s}^2
\end{equation}
that is obviously equivalent to the standard one.
Denote by $X^s_T$ the space of continuous functions
defined on $[0, T]$ with values in $X^s$,
equipped with the supremum-norm.
Define matrices
\[
	\mathcal K
	=
	\frac 1{\sqrt 2}
	\begin{pmatrix}
		1 & 1
		\\
		K & -K
	\end{pmatrix}
	, \quad
	\mathcal K^{-1}
	=
	\frac 1{\sqrt 2}
	\begin{pmatrix}
		1 & K^{-1}
		\\
		1 & -K^{-1}
	\end{pmatrix}
	.
\]
Clearly, that $\mathcal K$ is isometric from
$H^s \times H^s$ to $X^s$ for any $s \in \mathbb R$, i. e.
\(
	\lVert \mathcal K (f, g)^T \rVert_{X^s}
	= \lVert (f, g) \rVert_{H^s \times H^s}
\).
Regard the unitary group
\[
	\mathcal S(t) =
	\mathcal K
	\begin{pmatrix}
		e^{-itm} & 0
		\\
		0 & e^{itm}
	\end{pmatrix}
	\mathcal K^{-1}
\]
where $m = m(D) = \sqrt{D \tanh D} \sgn D$.
Note that for any $s, t \in \mathbb R$,
$u \in X^s$ holds
\(
	\lVert \mathcal S(t)u \rVert_{X^s}
	=
	\lVert u \rVert_{X^s}
\)
and consequently
\(
	\lVert \mathcal S(t)u \rVert_{X^s_T}
	=
	\lVert u \rVert_{X^s_T}
\)
for any $T>0$.
These follow from isometricity of operators
$\mathcal K$, $\mathcal K^{-1}$ and that symbols of
eigenvalues of $\mathcal S(t)$ have absolute value
equal to one.
For any fixed $u_0 = (\eta_0, v_0)^T \in X^s$
function $\mathcal S(t)u_0$ solves the linear
initial-value problem associated with
\eqref{wt1}.
Regard a mapping $\mathcal A : X^s_T \to X^s_T$
defined by
\begin{equation}
\label{contraction_mapping}
	\mathcal A(\eta, v) =
	\mathcal A(\eta, v; u_0)(t) = \mathcal S(t)u_0
	+ \int_0^t \mathcal S(t - t') (-i\tanh D)
	\begin{pmatrix}
		\eta v
		\\
		v^2 / 2
	\end{pmatrix}
	(t')dt'
	.
\end{equation}
Then the Cauchy problem for System \eqref{wt1}
with the initial data $u_0$ may be rewritten equivalently
as an equation in $X^s_T$ of the form
\begin{equation}
\label{u_is_Au}
	 u = \mathcal A(u; u_0)
\end{equation}
where $u = (\eta, v)^T \in X^s_T$.
Below the latter integral equation is solved locally in time
by making use of Picard iterations.
\begin{lemma}
[Particular case of Theorem \ref{simple_theorem}]
\label{particular_case_lemma}
	Let $s > 0$, $u_0 = (\eta_0, v_0)^T \in X^s$
	and $T = ( 7C_s \lVert u_0 \rVert_{X^s} )^{-1}$
	with some constant $C_s > 0$ depending only on $s$.
	Then there exists a unique solution
	$u = (\eta, v)^T \in X^s_T$ of Problem \eqref{u_is_Au}.

	Moreover, for any $R > 0$ there exists $T = T(R) > 0$
	such that the flow map associated with Equation \eqref{u_is_Au}
	is a real analytic mapping of the open ball
	$B_R(0) \subset X^s$ to $X^s_T$.
\end{lemma}
\begin{proof}
The idea is to show that the restriction of $\mathcal A$
on some closed ball $B_M$ centered at $\mathcal S(t)u_0$
is a contraction mapping.
The key ingredient is the product estimate
\(
	\norm{\eta v}_{H^s} \lesssim
	\norm{\eta}_{H^s} \norm{v}_{H^{s + 1/2}} 
\)
that can be found, for example in \cite{Hormander}.
Obviously, there exists a positive constant $C_s$
such that 
\[
	\lVert ( \eta v, v^2/2 ) \rVert _{X^s}
	\leqslant
	C_s \lVert ( \eta, v ) \rVert _{X^s}^2
\]
and
\[
	\lVert ( \eta_1 v_1 - \eta_2 v_2,
	v_1^2/2 - v_2^2/2 ) \rVert _{X^s}
	\leqslant
	C_s \lVert ( \eta_1 - \eta_2, v_1 - v_2 ) \rVert _{X^s}
	(
		\lVert ( \eta_1, v_1 ) \rVert _{X^s}
		+
		\lVert ( \eta_2, v_2 ) \rVert _{X^s}
	)
	.
\]

Thus for any $T, M > 0$ and $u, u_1, u_2 \in B_M \subset X^s_T$ hold
\[
	\lVert \mathcal A(u) - \mathcal S(t)u_0 \rVert _{X^s_T}
	\leqslant
	\int_0^T \lVert ( \eta v, v^2/2 ) \rVert _{X^s}
	\leqslant
	C_sT \lVert u \rVert _{X^s_T}^2
	,
\]
\[
	\lVert \mathcal A(u_1) - \mathcal A(u_2) \rVert _{X^s_T}
	\leqslant
	C_sT \lVert u_1 - u_2 \rVert _{X^s_T}
	(
		\lVert u_1 \rVert _{X^s_T}
		+
		\lVert u_2 \rVert _{X^s_T}
	)
	,
\]
and so taking $M = 2\lVert u_0 \rVert _{X^s}$
and $T$ as in the lemma formulation
we conclude that $\mathcal A$ is a contraction in
the closed ball $B_M$.
The first statement of the lemma follows from the
contraction mapping principle.

We turn our attention to smoothness of the flow map.
Let $R > 0$, $T = ( 7C_s R )^{-1}$ and
$B = B_R(0)$ be an open ball in $X^s$.
Define $\Lambda : B \times X^s_T \to X^s_T$ as
\[
	\Lambda(u_0, u) = u - \mathcal A(u; u_0)
\]
that is obviously a smooth map.
Its Fr\'echet derivative with respect to the second variable
is defined by
\[
	d_u\Lambda(u_0, u)h =
	h + i \int_0^t \mathcal S(t - t') \tanh D
	\begin{pmatrix}
		v & \eta
		\\
		0 & v
	\end{pmatrix}
	h(t')dt'
\]
where $u = (\eta, v)^T$ and $h \in X^s_T$.
If $u_1 \in X^s_T$ is the solution of Problem \eqref{u_is_Au}
corresponding the initial data $u_0 \in B$ then
$\Lambda(u_0, u_1) = 0$.
Moreover, it satisfies the following estimate
\[
	\lVert u_1(t) \rVert_{X^s}
	\leqslant
	\lVert u_0 \rVert_{X^s} +
	C_s \int_0^t \lVert u_1(t') \rVert_{X^s}^2 dt'
\]
and so
\[
	\int_0^t \lVert u_1(t') \rVert_{X^s}^2 dt'
	\leqslant
	\frac{ t \lVert u_0 \rVert_{X^s}^2 }
	{1 - C_st \lVert u_0 \rVert_{X^s}}
\]
for any $t$.
The latter is used to estimate operator
$ I - d_u\Lambda(u_0, u_1) $ as follows
\begin{multline*}
	\lVert h - d_u\Lambda(u_0, u_1)h \rVert
	\leqslant
	C_s \sup _{t \in [0, T]} \int_0^t
	\lVert u_1(t') \rVert_{X^s}
	\lVert h(t') \rVert_{X^s} dt'
	\\
	\leqslant
	C_s \sup _{t \in [0, T]}
	\left(
		t \int_0^t \lVert u_1(t') \rVert _{X^s}^2 dt'
	\right) ^{1/2}
	\lVert h \rVert_{X^s_T} 
	\leqslant
	\frac{ C_sT \lVert u_0 \rVert_{X^s} }
	{ \sqrt{ 1 - C_sT \lVert u_0 \rVert_{X^s} } }
	\lVert h \rVert_{X^s_T} 
	\leqslant
	\frac 1
	{ \sqrt{ 42 } }
	\lVert h \rVert_{X^s_T} 
\end{multline*}
which is true for any $h \in X^s_T$.
As a result operator $d_u\Lambda(u_0, u_1)$ is invertible
and so the second assertion of the lemma follows from
the Implicit Function Theorem.
\end{proof}

The next and most difficult step is to extend the statement
of the lemma to the case $s \leqslant 0$ as well.
Even extension to the limiting case $s = 0$ is not trivial.
On the one hand, it seems possible to do it without
the dispersive estimates,
applying the energy method, for example.
Indeed, we have the Hamiltonian conservation
that can provide us with a necessary a priori bound
(see Lemma \ref{global_lemma_s_0} below).
However, at such level of regularity with $s = 0$
the regularization of System \eqref{wt1}
can be a serious issue.
In other words, one cannot guarantee that the a priori
estimate will be still valid for the regularised problem.
Moreover, we can hardly hope for more than a weak solution
after implementing the compactness argument.
So we turn our attention to the Harmonic Analysis methods,
since we can eventually achieve a more general result
with the dispersive estimates obtained below
in the next sections.

\subsection
{
	Local well-posedness for $s>1/2$ in 2d
}
The proof is essentially the same.
Now the change of variables has the form
\[
	\mathcal K
	=
	\frac 1{\sqrt 2}
	\begin{pmatrix}
		1 & 1
		\\
		-iKR_1 & iKR_1
		\\
		-iKR_2 & iKR_2
	\end{pmatrix},
\]
where $K = \sqrt{\tanh |D| / |D|}$.
Then $\mathcal K$
is an isometric operator from
$H^s \times H^s$ to the subspace $X^s$ of
$H^s \times \left( H^{s+1/2} \right)^2$
with the curl free second coordinate
and endowed with the norm
\(
	\left \lVert
		\mathcal K^{-1} (\eta, \mathbf v)^T
	\right \rVert _{H^s \times H^s}
	.
\)
This $\mathcal K$ defines a continuous group
$\mathcal S(t)$ as above.
For any fixed $u_0 = (\eta_0, \mathbf v_0)^T \in X^s$
function $\mathcal S(t)u_0$ solves the linear
initial-value problem associated with
\eqref{wt2d}
in $X^s_T = C([0, T]; X^s)$.
Considering the map $\mathcal A : X^s_T \to X^s_T$
defined by
\begin{equation}
\label{contraction_mapping2d}
	\mathcal A(\eta, \mathbf v; u_0)(t) = \mathcal S(t)u_0
	- \int_0^t \mathcal S(t - t')
	\begin{pmatrix}
		K^2 \nabla \cdot ( \eta \mathbf v )
		\\
		K^2 \nabla \left( | \mathbf v|^2/2 \right)
	\end{pmatrix}
	(t')dt'
\end{equation}
we reduce
the Cauchy problem for System \eqref{wt2d}
with the initial data $u_0$ to Equation \eqref{u_is_Au}
in $X^s_T$ again, with the only difference that now
$u = (\eta, \mathbf v)^T \in X^s_T$ is a three component vector.
\begin{lemma}
[Particular case of Theorem \ref{simple_theorem2d}]
\label{particular_case_lemma2d}
	Let $s > 1/2$, $u_0 \in X^s$
	and $T = ( 7C_s \lVert u_0 \rVert_{X^s} )^{-1}$
	with some constant $C_s > 0$ depending only on $s$.
	Then there exists a unique solution
	$u \in X^s_T$ of Problem \eqref{u_is_Au}.

	Moreover, for any $R > 0$ there exists $T = T(R) > 0$
	such that the flow map associated with Equation \eqref{u_is_Au}
	is a real analytic mapping of the open ball
	$B_R(0) \subset X^s$ to $X^s_T$.
\end{lemma}
As above the key ingredient is the same product estimate
that in the case $d = 2$ is valid only provided $s > 1/2$,
and so we omit the proof.

\subsection
{
	A priori estimates for $s \geqslant 0$ in 1d
}
Firstly, we prove the following global bound
in the energy space $X^0$.

\begin{lemma}
\label{global_lemma_s_0}
	There exists a constant $\epsilon_0 > 0$ such that
	for any $\epsilon \in (0, \epsilon_0]$, if
	a pair
	\(
		u(t) = ( \eta(t), v(t) )
		\in L^2(\mathbb{R}) \times H^{1/2}(\mathbb{R})
	\)
	having initial condition
	\(
		\lVert u_0 \rVert _{L^2 \times H^{1/2}}
		\leqslant \epsilon / 2
	\),
	solves System \eqref{wt1},
	then its norm remains bounded
	\(
		\lVert u(t) \rVert _{L^2 \times H^{1/2}}
		\leqslant \epsilon
	\)
	for any time $t$.
\end{lemma}
\begin{proof}
We use a continuity argument.
Without loss of generality
we prove the statement with the $X^0$-norm defined in \eqref{Xs_norm_definition}, which is equivalent to the $L^2 \times H^{1/2}$-norm. For $u=(\eta, v)$, define
\[
	\lVert u \rVert ^2:
	=
	\frac12 \lVert u \rVert _{X^0}^2
	=
	\frac 12 \lVert \eta \rVert _{L^2}^2
	+
	\frac 12 \lVert K^{-1} v \rVert _{L^2}^2
	.
\]
Then there exists $C > 0$ such that
\[
	\lVert u \rVert ^2
	( 1 - C \lVert u \rVert )
	\leqslant
	\mathcal H(u)
	\leqslant
	\lVert u \rVert ^2
	( 1 + C \lVert u \rVert )
\]
where $u = u(t)$ is a solution of \eqref{wt1}
defined on some interval.
Take $\epsilon_0 = (2C)^{-1}$,
any $0 < \epsilon \leqslant \epsilon_0$
and a solution with $u_0 = u(0)$
having $\lVert u_0 \rVert \leqslant \epsilon / 2$.
By continuity
$\lVert u \rVert \leqslant \epsilon$
on some $[0, T_{\epsilon}]$ and so
\[
	\lVert u \rVert
	\leqslant
	\sqrt{2 \mathcal H(u)} = \sqrt{2 \mathcal H(u_0)}
	\leqslant
	\sqrt{\frac {1 + C \epsilon /2}2 } \epsilon < \epsilon
\]
which means that
the continuous function
$\lVert u(t) \rVert$ cannot touch the level
$\epsilon$ with time.
\end{proof}

Proving the next lemma we will employ
a sharper variant of the bilinear estimates
used at the beginning of the proof of Lemma \ref{particular_case_lemma}. 
Recall the notation
\(
	\lVert (\eta, v) \rVert _{X^s}
\)
defined by \eqref{Xs_norm_definition}.

\begin{lemma}
[Persistence of regularity]
\label{persistence_lemma}
	Suppose $s > 0$ and a pair
	$\eta(t) \in H^s$, $v(t) \in H^{s + 1/2}$
	solves Problem \eqref{wt1}, \eqref{data1}.
	Then if $s < 1/2$ the following holds true
	\[
		\lVert (\eta, v)(t) \rVert _{X^s}
		\leqslant
		\lVert (\eta_0, v_0) \rVert _{X^s}
		+ C_s \int_0^t
		\left(
			\lVert v \rVert _{H^{1/2}}
			+ \lVert v \rVert _{L^{\infty}}
		\right)
		\lVert (\eta, v) \rVert _{X^s}
		,
	\]
	and if $s \geqslant 1/2$ then
	\[
		\lVert (\eta, v)(t) \rVert _{X^s}
		\leqslant
		\lVert (\eta_0, v_0) \rVert _{X^s}
		+ C_s \int_0^t
		\lVert v \rVert _{H^{s + 1/4}}
		\lVert (\eta, v) \rVert _{X^s}
	\]
	where constant $C_s$ depends only on $s$.
\end{lemma}

\begin{proof}
Estimating $\mathcal A(t)$ given by
\eqref{contraction_mapping} in $X^s$-norm
defined by \eqref{Xs_norm_definition}, one deduces
from \eqref{u_is_Au} the following inequality
\[
	\lVert (\eta, v)(t) \rVert _{X^s}
	\leqslant
	\lVert (\eta_0, v_0) \rVert _{X^s}
	+ \int_0^t
	\left \lVert
		\begin{pmatrix}
			\tanh D ( \eta v )
			\\
			\tanh D ( v^2 / 2 )
		\end{pmatrix}
		(t')
	\right \rVert _{X^s} dt'
	.
\]
It is left to calculate the integrand.
Provided $s \in (0, 1/2)$ by the Leibniz rule
\cite{Linares_Ponce} we have
\begin{equation}
\label{persistence01}
	\lVert J^s \tanh D(\eta v) \rVert _{L^2}
	\lesssim
	\lVert J^s \eta \rVert _{L^{p_1}}
	\lVert v \rVert _{L^{q_1}}
	+
	\lVert \eta \rVert _{L^{p_2}}
	\lVert J^s v \rVert _{L^{q_2}}
\end{equation}
where setting $p_1 = 2$, $q_1 = \infty$, $p_2 = 2 / (1 - 2s)$,
$q_2 = 1/s$ and using the Sobolev embedding obtain
\[
	\lVert J^s \tanh D(\eta v) \rVert _{L^2}
	\lesssim
	\lVert \eta \rVert _{H^s}
	\left(
		\lVert v \rVert _{L^{\infty}}
		+
		\lVert v \rVert _{H^{1/2}}
	\right)
	.
\]
Similarly, but now for any $s \in (0, \infty)$ we have
\begin{equation}
\label{persistence02}
	\left \lVert J^s K^{-1} \tanh D v^2 \right \rVert _{L^2}
	\lesssim
	\left \lVert J^{s + 1/2} v^2 \right \rVert _{L^2}
	\lesssim
	\lVert v \rVert _{L^{\infty}}
	\left \lVert J^{s + 1/2} v \right \rVert _{L^2}
	\lesssim
	\lVert v \rVert _{L^{\infty}}
	\left \lVert K^{-1} v \right \rVert _{H^s}
	.
\end{equation}
This implies the first inequality in the statement
valid for $s \in (0, 1/2)$.

Regarding the case $s = 1/2$ and setting
$p_2 = q_2 = 4$ with the same
$p_1 = 2$, $q_1= \infty$
in the Leibniz inequality \eqref{persistence01},
after implementation the Sobolev embedding,
obtain
\[
	\lVert J^s \tanh D(\eta v) \rVert _{L^2}
	\lesssim
	\lVert \eta \rVert _{H^s}
	\lVert v \rVert _{H^{s + 1/4}}
	.
\]
This inequality is obvious for $s > 1/2$
since $H^s$
is an algebra under the point-wise product,
and so is true for any $s \geqslant 1/2$.
Taking into account \eqref{persistence02}
we deduce the second inequality of the lemma.
\end{proof}

In order to use the persistence of regularity lemma
\ref{persistence_lemma} one needs
two Gronwall inequalities.
One of them is considered to be standard.
For the completeness,
we give here a proof of the other
Gronwall type inequality, which is less standard
and will be used below.

\begin{lemma}
[Gronwall inequality]
\label{Gronwall_lemma}
	Let $y(t) > 1$ be a continuous 
	function defined on some interval $[0, T]$
	with $y(0) = y_0$.
	Suppose that for any $t \in [0, T]$ hold
	\[
		y(t) \leqslant y_0 + C \int_0^t y \log y
		.
	\]
	Then
	\[
		y(t) \leqslant \exp \left( e^{Ct} \log y_0 \right)
		.
	\]
\end{lemma}

\begin{proof}
One can easily calculate
\[
	\frac d{dt} \log \log
	\left(
		y_0 + C \int_0^t y \log y
	\right)
	=
	\frac
	{
		C y \log y
	}
	{
	\left(
		y_0 + C \int_0^t y \log y
	\right)
	\log
	\left(
		y_0 + C \int_0^t y \log y
	\right)
	}
	\leqslant C
\]
where we have used the dominance of $y(t)$
by the integral expression.
The fundamental theorem of calculus
provides us with the claim.

\end{proof}

The persistence of regularity
based on the energy estimate lemma \ref{persistence_lemma}
transforms to the following a priori estimates.

\begin{lemma}
\label{global_lemma_s_positive}
	Suppose $s > 0$ and a pair
	$u(t) = ( \eta(t), v(t) ) \in X^s$ solves
	System \eqref{wt1} on some time interval
	with $u(0) = u_0$ small enough
	with respect to $X^0$-norm
	in the sense of Lemma \ref{global_lemma_s_0}.
	Then if $s < 1/2$ the following holds true
	\[
		\lVert u(t) \rVert _{X^s}
		\leqslant
		\exp
		\left(
			Ce^{Ct}
		\right)
		,
	\]
	and if $s \geqslant 1/2$ then
	\[
		\lVert u(t) \rVert _{X^s}
		\leqslant
		\lVert u_0 \rVert _{X^s}
		\exp
		\left(
			C \int_0^t \lVert v \rVert _{H^{s + 1/4}}
		\right)
	\]
	where constant $C$ depends only on $s$,
	$\lVert u(0) \rVert _{X^0}$ and
	$\lVert u(0) \rVert _{X^s}$.
\end{lemma}

\begin{proof}
Suppose $s \in (0, 1/2)$ and
$u(t) = ( \eta(t), v(t) ) \in X^s$ solves
System \eqref{wt1} on some time interval.
Let its initial data $u_0$ be small
with respect to $X^0$-norm
in the sense of Lemma \ref{global_lemma_s_0}.
Then $u(t)$ stays bounded in $X^0$,
and so $\lVert v(t) \rVert _{H^{1/2}}$ is bounded
by the same constant independent on the time interval.
Hence from the Brezis-Gallouet limiting embedding
\eqref{Brezis_inequality} one deduces
\[
	\lVert v(t) \rVert _{L^{\infty}}
	\lesssim 1 +
	\log \left( 2 + \lVert v(t) \rVert _{H^{s + 1/2}} \right)
\]
and applying Lemma \ref{persistence_lemma} obtain
\[
	\lVert u \rVert _{X^s}
	\leqslant
	\lVert u_0 \rVert _{X^s}
	+ C \int_0^t
	\left(
		1 +
		\log \left( 2 + \lVert u \rVert _{X^s} \right)
	\right)
	\lVert u \rVert _{X^s}
	.
\]
Introducing
\(
	y(t) = 2 + \lVert u(t) \rVert _{X^s}
\)
we arrive at the assumption of 
the Gronwall inequality lemma \ref{Gronwall_lemma}.
As a result we have the estimate
\[
	2 + \lVert u \rVert _{X^s}
	\leqslant
	\exp
	\left(
		e^{2Ct} \log \left( 2 + \lVert u_0 \rVert _{X^s} \right)
	\right)
\]
that is the first claim.

In the case $s \geqslant 1/2$ we make use of
the second inequality in Lemma \ref{persistence_lemma}
and a more standard Gronwall inequality \cite{t06}.
\end{proof}

\section{Dispersive estimate for $S_{m_d}(\pm t) f$ }
\label{free_waves_estimates}
\setcounter{equation}{0}

First we establish a lower bound for the first
and second derivatives of the function $m(r)=r\sqrt{\tanh(r)/r}$.
These estimates will be used later to derive dispersive estimates for
the free waves $S_{m_d}(\pm t) f$ using a stationary phase method.

Throughout the next three sections we use the following notation:
The Greek letter $\lambda$ denotes a dyadic number, i.e., this
variable ranges over numbers of the form $2^k$ for $k \in \Z$.
 In estimates we use $A\lesssim B$ as shorthand for $A \le CB$ and $A\ll B$ for $A\le C^{-1} B$, where $C\gg1 $ is a positive constant which is independent of dyadic numbers such as $\lambda$ and time $T$, whereas $A\sim B$ means $B\lesssim A\lesssim B$. 
 
\begin{lemma}\label{lm-mest}
Set $m(r)=r K(r)$, where $K(r)=\sqrt{\tanh(r)/r}$. Then for $r > 0$,
\begin{align}
\label{m'est} 
0 < m'(r)& \sim \angles{r}^{-1/2},
 \\
0 < - m''(r) &\sim  r \angles{r}^{-5/2}.
\end{align}
\end{lemma}

\begin{proof}
First note that
\begin{align*}
K'(r)&= \frac{ r \text{sech}^2(r)-\tanh(r)}{2 r^2 K(r)},
\\
K''(r)&= -\frac{ \tanh(r)\text{sech}^2(r)}{ r K(r)}- \frac{ \left( r \text{sech}^2(r)-\tanh(r) \right)}{ r^3K(r)}- \frac{ \left( r \text{sech}^2(r)-\tanh(r) \right)^2  }{4 r^4 K^3(r)}
\end{align*}
which imply
\begin{align*}
m'(r)&= K(r)+rK'(r)=\frac{K(r)}{2}+ \frac{\text{sech}^2(r)}{2K(r)} > 0,
\\
m''(r)&= 2K'(r)+ rK''(r)
\\
&=-\frac{ \tanh(r)\text{sech}^2(r)}{ K(r)}- \frac{ \left( r \text{sech}^2(r)-\tanh(r) \right)^2  }{4 r^3 K^3(r)}
\\
&=-\frac 1{4r} \left[    4r^2K\text{sech}^2(r) +  K^{-3}(r)\left( K^2(r)-\text{sech}^2(r) \right)^2 \right].
\end{align*}

Now let us estimate $m'(r)$.
One can assume without loss of generality
that $r > 0$.
Since $$K(r)=\sqrt{\tanh(r)/r}\sim  \angles{r}^{-1/2} \quad \text{and} \quad  \text{sech}(r)\sim e^{-r}$$
we have
\begin{equation}\label{KS-Est}
m'(r) \sim  \angles{r}^{-1/2}+  \angles{r}^{1/2} e^{-2r}  \sim \angles{r}^{-1/2}.
\end{equation}

Next we estimate $m''(r)$.  
We can write
\begin{align*}
K^2(r)-\text{sech}^2(r)=  E(r)\text{sech}^2(r),
\end{align*}
where
$$E(r)=\frac{e^{2r}- e^{-2r}-4r} {4r}.$$
Now if $0<r<1/2$ we have
$$E(r)=\frac{1}{2r}\sum_{n=0}^\infty \frac{(2r)^{2n+3}}{(2n+3)!} =4Cr^2, $$
where 
$C:=C(r)=\sum_{n=0}^\infty \frac{(2r)^{2n}}{(2n+3)!} <\infty$. If $r\ge 1/2$  we have
$$
E(r)=\frac{e^{2r}} {4r} [1-e^{-4r}-4r e^{-2r} ] \sim \frac{e^{2r}} {r}.
$$
 Therefore, 
\begin{equation} \label{E-est}
E(r) \sim 
\begin{cases}
& r^{2}\qquad \text{if} \ 0<r<1/2,
\\
& r^{-1}e^{2r}  \qquad \text{if} \ r\ge 1/2.
\end{cases}
\end{equation}
Then using \eqref{KS-Est} and \eqref{E-est} we obtain
\begin{align*}
|m''(r)|
&=\frac 1{4|r|} \left[    4r^2K(r)\text{sech}^2(r) +  K^{-3}(r)E^2(r) \text{sech}^4(r) \right]
\\
&\sim |r|^{-1} \left[   r^2 \angles{r}^{-\frac12}  e^{-2r}+  \angles{r}^{\frac32}E^2(r) e^{-4r} \right]
\\
&\sim |r| \angles{r}^{-5/2} .
\end{align*}

\end{proof}

Next we use the estimates on the derivatives of $m(r)$ in Lemma \ref{lm-mest} and stationary phase method to derive a frequency localized dispersive estimate for the free waves $S_{m}(\pm t) f$. 
To this end, we consider an even function $ \chi \in C_0^\infty((-2, 2))$ such that 
$\chi (s)=1 $ if $|s|\le 1$. Let
\begin{align*}
\beta(s)
=\chi\left(s\right)-\chi \left(2s\right), \qquad \beta_{\lambda}(s):=\beta\left(s/\lambda\right),
 \end{align*}
 where $\lambda \in 2^\Z$ is dyadic.
 Thus, $\supp \beta_\lambda \subset 
\{ s\in \R: \lambda/ 2 \leqslant |s| \leqslant 2\lambda \}$.
 Now define the frequency projection $P_\lambda$ by
\begin{align*}
\widehat{P_{\lambda} f}(\xi)  &=\begin{cases}
& \chi (|\xi|)\widehat { f}(\xi) \quad \text{if} \ \lambda=1,
\\
& \beta_\lambda(|\xi|)\widehat { f}(\xi) \quad \text{if} \ \lambda>1.
\end{cases}
 \end{align*}
We write $f_\lambda:=P_\lambda f $.
Then $f=\sum_{\lambda \ge 1 } f_\lambda$.

The following is the key dispersive estimate that will be crucial in the proof of Theorem \ref{lwpthm} and Theorem \ref{lwpthm2d-1}.
\begin{lemma}[Localized dispersive estimate]
\label{lm-dispest}
Let $\lambda >1$ and $d\in \{1, 2\}$. Then we have the estimate
\[
\| S_{m_d}(\pm t) f_\lambda  \|_{L^\infty_x(\R^d)} \lesssim  \lambda^{3d/4} |t|^{-d/2}  \|f\|_{L_x^1(\R^d)}.
\]
\end{lemma}

Interpolating this with the trivial bound (by Plancherel)
\[
\| S_{m_d}(\pm t) f_\lambda  \|_{L^2_x(\R^d)}  \le \|f\|_{L_x^2(\R^d)},
\]
we obtain the following.

\begin{corollary}\label{dispcor} Assuming $\lambda >1$, $d\in \{1, 2\}$ and $2 \le r \le \infty$, we have
\[
\| S_{m_d}(\pm t) f_\lambda  \|_{L^r_x(\R^d)} \lesssim  \left( \lambda^{3d/4} |t|^{-d/2} \right)^{1-2/r}  \|f\|_{L_x^{r'}(\R^d)}.
\]
\end{corollary}

The remainder of this section is devoted to the proof of Lemma \ref{lm-dispest}. It suffices to prove the estimate for positive times:
\begin{equation} 
\label{disest1}
\| S_{m_d}( t) f_\lambda  \|_{L^\infty_x(\R^d)} \lesssim \lambda^{3d/4} t^{-d/2 }     \|f\|_{L_x^1(\R^d) } \quad (t>0).
\end{equation}
One can write
\[
\left[ S_{m_d}( t) f_\lambda \right](x) = \mathcal F_x^{-1}\left[e^{ it m_d(\xi) } \beta_\lambda(|\xi|) \hat f\right](x)
= ( I_{\lambda, t} \ast f)(x),
\]
where
\begin{equation}\label{Idef}
 I_{\lambda, t} (x)=\mathcal F_x^{-1}\left[e^{i t m_d(\xi) } \beta_\lambda(|\xi|) \right](x) = \int_{\R^d} e^{ix\cdot \xi+ it m_d(\xi) }  \beta_\lambda(|\xi|)  \, d\xi 
=\lambda^d \int_{\R^d} e^{i \lambda x\cdot \xi+ it {m_d}(\lambda \xi)}  \beta(|\xi|) \, d\xi .
\end{equation}
Then by Young's inequality
\begin{equation}\label{younginq}
\| S_{m_d}(t) f_\lambda  \|_{L^\infty_x(\R^d)} \le \| I_{\lambda, t} \|_{L^\infty_x(\R^d)}  \|f\|_{L_x^1(\R^d)},
\end{equation}
so \eqref{disest1} reduces to proving
\begin{equation}
\label{Iest-1}
\|I_{\lambda, t} \|_{L^\infty_x(\R^d)}  \lesssim \lambda^{3d/4}  t^{-d/2 }   .
\end{equation}
But clearly,
\begin{equation*}
\|I_{\lambda, t} \|_{L^\infty_x(\R^d)} 
\lesssim  \lambda^d,
\end{equation*}
so in view of \eqref{Iest-1} it is enough to consider the case where 
\begin{equation}
\label{tlower1}
     \lambda ^{3d/4} t^{-d/2}\ll \lambda^d \Leftrightarrow  t \gg     \lambda ^{-1/2} .
 \end{equation}
The proof of \eqref{Iest-1} in this case is given in the following two subsections, first for space dimension $d=1$ and then for $d=2$. 
\subsection{Proof of \eqref{Iest-1} when $d=1$  }
In one dimension we have
\begin{align*}
I_{\lambda, t} (x)
=\lambda \int_{\R} e^{i t\phi_{\lambda }(\xi) }  \beta(|\xi|) \, d\xi,
\end{align*}
where
$$\phi_{\lambda}(\xi):= \lambda  \xi x/t  +  m_1(\lambda \xi)=\lambda  \xi  x/t     +  \lambda \xi K_1(  \lambda \xi ).$$
Note that $m_1(\xi)=m(\xi)$, where $m$ is as in Lemma \ref{lm-mest}.
Now since the function $\phi_{\lambda}$ is odd we can write
\begin{align*}
I_{\lambda, t } (x)
 &=2\lambda \int_{0}^\infty \cos ( t\phi_{\lambda }(\xi) )  \beta(\xi) \, d\xi=2\lambda \int_{1/2}^2 \cos ( t\phi_{\lambda }(\xi) )  \beta(\xi) \, d\xi.
\end{align*}
Since
\begin{align}
\label{ph'}
\phi_{\lambda}'(\xi)&=  \lambda  \left[  x/t  +  m'(\lambda \xi)\right] ,
\\
\label{ph''}
  \phi_{\lambda }''(\xi)&=     \lambda^2 m''(\lambda \xi),
\end{align}
we see from Lemma \ref{lm-mest} that
\begin{equation}
\label{Est-phi''-1}
0 < - \phi_{\lambda }''(\xi) = - \lambda^2 m''(\lambda \xi) \sim   \lambda^3\angles{  \lambda }^{-5/2}\sim \lambda^{1/2}
\end{equation}
for $\xi \in [1/2,2]$. Here we used also the assumption $\lambda \ge 1$.

\subsubsection{Non-stationary contribution} This is the case when either (i) $x\ge 0$ or (ii) $x<0$ and $-x/t \ll    \lambda ^{-1/2}$ or $-x/t \gg   \lambda ^{-1/2}$.
Then since $m'(\lambda \xi)) $ is positive and comparable to $\angles{\lambda\xi}^{-1/2}$ (Lemma \ref{lm-mest}), we see from \eqref{ph'} that
\begin{equation}
\label{Est-phi'-1}
|\phi_{\lambda}'(\xi)|\gtrsim   \lambda ^{1/2}
\end{equation}
for $\xi \in [1/2,2]$.
Integration by parts yields
\begin{equation}
\label{Ilam-int}
\begin{split}
I_{\lambda, t} (x)
&= 2\lambda t^{-1}\int_{1/2}^2  \frac{d}{d\xi}\left[\sin ( t\phi_{\lambda}(\xi) ) \right] \left[ \phi_\lambda'(\xi) \right]^{-1}\beta(\xi) \, d\xi
\\
&=- 2\lambda t^{-1} \int_{1/2}^2  \sin \left(t\phi_{\lambda}(\xi) \right) [\phi_{\lambda}'(\xi)]^{-2}\left[\beta'(\xi)\phi_{\lambda}'(\xi)- 
\beta(\xi) \phi_{\lambda}''(\xi) \right]  \, d\xi,
\end{split}
\end{equation}
hence \eqref{Est-phi''-1} and \eqref{Est-phi'-1} allow us to estimate
\begin{equation}\label{Iest-3}
\begin{split}
|I_{\lambda, t} (x)|
&\le 2\lambda t^{-1}\int_{1/2}^2  |\phi_{\lambda}'(\xi)|^{-2} \left[|\beta'(\xi)||\phi_{\lambda}'(\xi)|+ |\beta(\xi)||\phi_{\lambda}''(\xi)| \right]  \, d\xi  
\\
&\lesssim \lambda t^{-1} \left[  \lambda^{-1/2}  +   \lambda^{-1} \cdot \lambda^{1/2}  \right]
\\
&\sim    \lambda^{1/2}t^{-1}
\\
&\ll \lambda^{3/4} t^{-1/2},
\end{split}
\end{equation}
where the last step follows by the assumption \eqref{tlower1}. This concludes the proof of the desired estimate \eqref{Iest-1} with $d=1$ in the non-stationary case.

\subsubsection{ Stationary contribution: $x < 0$ and $-x/t \sim \lambda ^{-1/2}$ }\label{1Dstationary} In this case, we see from \eqref{ph'} that $\phi_\lambda'(\xi)$ may vanish, but this can happen for at most one point $\xi \in [1/2,2]$, since $\xi \mapsto \phi_\lambda'(\xi)$ is strictly decreasing for $\xi > 0$ (indeed, $\phi_{\lambda }''(\xi)$ is negative, by Lemma \ref{lm-mest}). We consider first the case where there exists such a point in $[1/2,2]$. 

So suppose first that $\phi_\lambda'(\xi_0) = 0$ for some $\xi_0 \in [1/2,2]$. Define
\[
  \delta = t^{-1/2} \lambda^{-1/4}.
\]
Note that $\delta \ll 1$ by \eqref{tlower1}. Assuming for the moment that $1/2 \le \xi_0-\delta$ and $\xi_0+\delta \le 2$, we decompose the integral as
\begin{equation}
\label{Idecomp}
I_{\lambda, t}(x)
= 2\lambda \left( \int_{1/2}^{\xi_0-\delta} + \int_{{\xi_0-\delta}}^{{\xi_0+\delta}}+  \int_{{\xi_0+\delta}}^2 \right)\cos (  t\phi_{\lambda}(\xi) )  \beta(\xi) \, d\xi.
\end{equation}
To estimate the first integral, we use integration by parts to get
\begin{align*}
&\left| \int_{1/2}^{\xi_0-\delta} \cos (t \phi_{\lambda}(\xi) )  \beta(\xi) \, d\xi \right|
\\
& \quad \le t^{-1} \left| \left[  \sin ( t \phi_{\lambda}(\xi) )  \frac{\beta(\xi) }{\phi_{\lambda}'(\xi) } \right]_{\xi=1/2}^{\xi=\xi_0-\delta} \right| 
 +  t^{-1} \left|\int_{1/2}^{\xi_0-\delta}  \sin ( t \phi_{\lambda}(\xi) ) 
 \left( \frac{ \beta'(\xi)}{ \phi_{\lambda}'(\xi) } - \frac{ \beta(\xi)\phi_{\lambda}''(\xi)}{ [\phi_{\lambda}'(\xi)]^2 } \right) \, d\xi \right|.
\end{align*}
Since $\phi_\lambda'$ is positive and decreasing in the interval $[1/2,\xi_0-\delta]$, and since $\phi_{\lambda }''$ is negative, we can continue the estimate by
\begin{align*}
  &\lesssim  t^{-1} \left( \frac{1}{\phi_{\lambda}'(\xi_0-\delta)}
  + \int_{1/2}^{\xi_0-\delta} \frac{ -\phi_{\lambda}''(\xi)}{ [\phi_{\lambda}'(\xi)]^2 } \, d\xi\right)
  \\
  &=  t^{-1} \left( \frac{1}{\phi_{\lambda}'(\xi_0-\delta)}
  + \int_{1/2}^{\xi_0-\delta} \frac{d}{d\xi} \left( \frac{1}{ \phi_{\lambda}'(\xi) } \right) \, d\xi\right)
  \\
  &\le 2  t^{-1} \frac{1}{\phi_{\lambda}'(\xi_0-\delta)}.
\end{align*}
But by the mean value theorem and \eqref{Est-phi''-1},
\[
 |\phi_{\lambda}'(\xi)|=  |\phi_{\lambda}'(\xi)-\phi_{\lambda}'(\xi_0)|
 \sim \lambda ^{1/2} |\xi-\xi_0|
 \quad \text{for $\xi \in [1/2,2]$},
\]
so we we conclude that
\[
  \left| \int_{1/2}^{\xi_0-\delta} \cos (t \phi_{\lambda}(\xi) )  \beta(\xi) \, d\xi \right|
  \lesssim
   t^{-1} \lambda^{-1/2} \delta^{-1}
   =
   t^{-1/2} \lambda^{-1/4},
\]
by the definition of $\delta$ above. The third integral in \eqref{Idecomp} can be estimated in a similar way, and satisfies the same estimate, while the second integral \eqref{Idecomp} is trivially estimated as 
\begin{align*}
\int_{\xi_0-\delta}^{\xi_0+\delta} \cos ( t\phi_{\lambda}(\xi) )  \beta(\xi) \, d\xi  \lesssim \delta =  t^{-1/2} \lambda^{-1/4}.
\end{align*}

Summing up the three contributions, we conclude that the desired estimate holds,
\[
|I_{\lambda, t}(x)| \lesssim  \lambda^{3/4}  t^{-1/2},
\]
in the stationary case under the assumptions that $\phi_\lambda'(\xi_0) = 0$ for some $\xi_0 \in [1/2,2]$, and that $1/2 \le \xi_0-\delta$ and $\xi_0+\delta \le 2$. If $1/2 > \xi_0-\delta$ or $\xi_0+\delta > 2$, the above argument is easily modified. For example, if $\xi_0+\delta > 2$, we split the integral as $\int_{1/2}^{\xi_0-\delta} + \int_{\xi_0-\delta}^2$ instead; the first integral is then treated as above and the second is trivially $O(\delta)$.

It remains to prove the estimate when the function $\phi_\lambda'$ has no zero in $[1/2,2]$, so it is either positive or negative everywhere in that interval. Since the arguments for these two cases are similar, we just treat the case where $\phi_\lambda' < 0$ in $[1/2,2]$. Then we split the integral as
\[
I_{\lambda, t}(x)
= 2\lambda \left( \int_{1/2}^{1/2+\delta} + \int_{1/2+\delta}^{2-\delta} + \int_{2-\delta}^{2} \right)\cos (  t\phi_{\lambda}(\xi) )  \beta(\xi) \, d\xi.
\]
The first and third integrals are trivially dominated in absolute value by $\delta$, while for the second integral we use integration by parts, estimating
\begin{align*}
\left| \int_{1/2+\delta}^{2-\delta} \cos (t \phi_{\lambda}(\xi) )  \beta(\xi) \, d\xi \right|
&\lesssim  t^{-1} \left( \frac{1}{-\phi_{\lambda}'(1/2+\delta)}
  + \int_{1/2+\delta}^{2-\delta} \frac{ -\phi_{\lambda}''(\xi)}{ [\phi_{\lambda}'(\xi)]^2 } \, d\xi\right)
  \\
  &=  t^{-1} \left( \frac{1}{-\phi_{\lambda}'(1/2+\delta)}
  + \int_{1/2+\delta}^{2-\delta} \frac{d}{d\xi} \left( \frac{1}{ \phi_{\lambda}'(\xi) } \right) \, d\xi\right)
  \\
  &\le 2  t^{-1} \frac{1}{-\phi_{\lambda}'(1/2+\delta)}.
\end{align*}
Here we used the fact that $\phi_\lambda'$ is negative and decreasing in the interval $[1/2,2]$, and that $\phi_{\lambda }''$ is negative. Using the mean value theorem and the estimate \eqref{Est-phi''-1} on the second derivative, we find moreover that
\[
 - \phi_{\lambda}'(1/2+\delta) \ge \phi_{\lambda}'(1/2)  - \phi_{\lambda}'(1/2+\delta)
 \sim \lambda ^{1/2} \delta,
\]
so we conclude that
\[
  \left| \int_{1/2+\delta}^{2-\delta} \cos (t \phi_{\lambda}(\xi) )  \beta(\xi) \, d\xi \right|
  \lesssim
   t^{-1} \lambda^{-1/2} \delta^{-1}
   =
   t^{-1/2} \lambda^{-1/4},
\]
as desired.

\subsection{Proof of \eqref{Iest-1} when $d=2$  }
In two dimensions we have
$$
I_{\lambda, t} (x)=\lambda^2 \int_{\R^2} e^{i \lambda x \cdot \xi+it {m_2}(\lambda \xi)}  \beta(\xi) \, d\xi 
$$
which is the inverse Fourier transform of the radial function $\lambda^2 e^{it {m_2}(\lambda \xi)}  \beta(\xi) $, and hence $I_{\lambda, t}(x)$ is also radial.
So we may set $x =(|x|, 0)$.
Then in polar coordinates we have
\begin{align*}
 I_{\lambda, t}(x)
&=\lambda^2 \int_{0}^\infty  \int_0^{2\pi}  e^{i \lambda r |x| \cos\theta} e^{it m_2(\lambda r) } r \beta(r) \, d\theta dr.
\end{align*}
We can write 
\begin{align*}
  \int_0^{2\pi}  e^{i \lambda r |x| \cos\theta} \, d\theta &=  \int_0^{\pi}    \left(  e^{i \lambda r |x| \cos\theta} +e^{-i \lambda r |x| \cos\theta}   \right)\, d\theta 
  \\
  &=  2\int_{-1}^{1}      e^{i \lambda r |x| s} \left(1-s^2\right)^{-1/2} \, ds
  \\
  &=2\pi J_0( \lambda r |x|),
\end{align*}
where $J_k(r)$ is the Bessel function:
$$
J_k(r)=\frac{ (r/2)^k}{\Gamma(k+1/2) \sqrt{\pi}} \int_{-1}^1  e^{ir s} \left(1-s^2\right)^{k-1/2} \, ds \quad \text{for} \ k>-1/2.
$$
 Thus,
\begin{equation}\label{I-eq}
I_{\lambda, t} (x)
 =2\pi \lambda^2 \int_{1/2}^2  e^{it m(\lambda r) }J_0( \lambda r |x|) \tilde \beta(r) \, dr,
\end{equation}
 where $\tilde \beta(r) = r \beta(r)$ and $m(r)=m_2(r)$. 
 
 We shall use the following properties of $J_k(r)$ for $k>-1/2$ and $r>0$ (See   \cite[Appendix B]{G08} and \cite{S93}.)
\begin{align}
\label{Jm1}
J_k (r) &\le Cr^{k} ,
\\
\label{Jm2}
J_k(r)& \le C r^{-1/2} ,
\\
\label{Jm3}
\partial_r \left[ r^{-k} J_k(r)\right] &= -r^{-k} J_{k+1}(r)
\end{align}
Moreover, we can write
\begin{equation}
\label{J0est}
J_0(s)= e^{is} h(s)  +e^{-is}\bar h(s)
\end{equation}
for some function $h$ satisfying the estimate 
\begin{equation}
\label{h-est}
| \partial_r ^j h(r)|\le C_j \angles{r}^{-1/2-j}  \quad \text{for all} \ j\ge 0.
\end{equation}

We treat the cases $  |x|\lesssim \lambda^{-1}$ and $  |x|\gg  \lambda^{-1}$ separately. 
\subsubsection{Case 1: $  |x|\lesssim  \lambda^{-1}$}
By \eqref{Jm1} and \eqref{Jm3} we have for all $ r\in (1/2, 2)$ the estimate
\begin{equation}
\label{J0derv-est}
\left| \partial_r^j J_0( \lambda r |x|) \right|\lesssim 1  \quad \text{for $j=0,1$}.
\end{equation}
We integrate by parts \eqref{I-eq} to obtain
\begin{align*}
 I_{\lambda, t}(x)
&=-2\pi i\lambda t^{-1}  \int_{1/2}^2  \frac{d}{dr}\left\{ e^{it m (\lambda r) }\right\}  [ m'(\lambda r) ]^{-1} J_0( \lambda r |x|) \tilde\beta(r) \, dr
\\
&=2\pi i\lambda t^{-1}  \int_{1/2}^2   e^{it m (\lambda r) }[m'(\lambda r) ]^{-1} \partial_r\left[ J_0( \lambda r |x|) \tilde\beta(r) \right] \, dr
\\
 & \qquad -2\pi i\lambda t^{-1}  \int_{1/2}^2   e^{it m (\lambda r) } [m'(\lambda r) ]^{-2} \lambda m''(\lambda r) J_0( \lambda r |x|) \tilde\beta(r)  \, dr.
\end{align*}
Then applying Lemma \ref{lm-mest} 
and \eqref{J0derv-est} 
we obtain 
\begin{equation}
\label{I-est1-2d}
| I_{\lambda, t}(x)| \lesssim \lambda  t^{-1} \left( \lambda^{1/2} + \lambda^2 \cdot \lambda^{-3/2} \right)
\lesssim \lambda  ^{3/2} t^{-1} .
\end{equation}

\subsubsection{Case 2: $  |x|\gg  \lambda^{-1}$ }
Using \eqref{J0est} in \eqref{I-eq} we write
\begin{align*}
 I_{\lambda, t}(x)
 &=2\pi \lambda^2 \left\{\int_{1/2}^2  e^{it \phi^+_{\lambda} (r)  }  h(\lambda r |x|)  \tilde \beta(r) \, dr +  \int_{1/2}^2  e^{-it \phi^-_{\lambda} (r)  } \bar  h(\lambda r |x|)  \tilde\beta(r) \, dr \right\},
\end{align*}
where 
$$
\phi^\pm_{\lambda} (r)=    \lambda r|x|/t  \pm  m (\lambda r) .
$$
Set $F_{\lambda}( |x|, r) =h(\lambda r |x|)  \tilde \beta(r)$. In view of \eqref{h-est} we have 
\begin{equation}
\label{Fest}
| F_{\lambda}( |x|, r)  | +  |\partial_r  F_{\lambda}( |x|, r)  | \lesssim   (\lambda |x|)^{-1/2}
\end{equation}
for all $ r\in (1/2, 2)$, 
where we also used the fact $\lambda |x|\gg 1$.

Now 
we write 
$$
 I_{\lambda, t}(x)= I^+_{\lambda, t}(x) + I^-_{\lambda,  t}(x),
$$
where
\begin{align*}
  I^+_{\lambda, t}(x)
 &=2\pi \lambda^2 \int_{1/2}^2  e^{it \phi^+_{\lambda} (r)  } F_{\lambda}( |x|, r)  \, dr ,
 \\
 I^-_{\lambda, t}(x)&= 2\pi \lambda^2
   \int_{1/2}^2  e^{-it \phi^-_{\lambda} (r)  }\bar F_{\lambda}( |x|, r)  \, dr .
\end{align*}

Observe that
$$
\partial_r \phi^\pm_{\lambda} (r)=  \lambda  \left[ |x|/t \pm  m'(\lambda r) \right],\qquad \partial_r^2\phi^\pm_{\lambda} (r)=     \pm  \lambda^2 m''(\lambda r),
$$
and hence by Lemma \ref{lm-mest}, 
\begin{equation}
\label{phi'+:est}
|\partial_r \phi^+_{\lambda} (r)|\gtrsim  \lambda ^{1/2},
\qquad
|\partial^2_r \phi^\pm_{\lambda} (r)| \sim  \  \lambda^{1/2}
\end{equation}
for all $ r\in (1/2, 2)$, where we also used the fact that $m'$ is positive.

\subsubsection*{\underline{Estimate for  $I^+_{\lambda, t}(x)$ } }
It is easy to estimate  $I^+_{\lambda, t}(x)$ since $\partial_r \phi^+_{\lambda}(r)$ is never zero. Indeed, 
using
integration by parts we have
\begin{align*}
 I^+_{\lambda, t}(x)
&=- 2\pi i t^{-1} \lambda^2 \int_{1/2}^2 \partial_r\left[  e^{it \phi^+_{\lambda} (r)  } \right]  \left[\partial_r \phi^+_{\lambda} (r) \right]^{-1}   F_{\lambda}( |x|, r)  \, dr
\\
&=2\pi i t^{-1} \lambda^2  \int_{1/2}^2  e^{it \phi^+_{\lambda} (r)  }  \left\{
\frac{\partial_r F_{\lambda}( |x|, r)  }{\partial_r \phi^+_{\lambda} (r) } -   \frac{\partial^2_r \phi^+_{\lambda} (r)   F_{\lambda}( |x|, r)  }{\left[\partial_r \phi^+_{\lambda} (r) \right]^{2}}\right\}\, dr.
\end{align*}
Then using \eqref{Fest} and \eqref{phi'+:est} we have
\begin{equation}
\label{I-est2-2d}
| I^+_{\lambda, t}(x)|
\lesssim  t^{-1} \lambda^2 \cdot  \lambda^{-1/2} \cdot (\lambda |x|)^{-1/2}
\lesssim  \lambda  ^{3/2}   t^{-1} .
\end{equation}

\subsubsection*{\underline{Estimate for $I^-_{\lambda, t}(x)$}}
We treat the the non-stationary and stationary cases separately. In the non-stationary
case, where $ |x |/t \ll   \ \lambda ^{-1/2}$ or $|x|/t \gg   \lambda ^{-1/2}$, we have 
$$
|\partial_r \phi^-_{\lambda} (r)|\gtrsim \ \lambda^{1/2},
$$
and hence $I^-_{\lambda, t}(x)$ can be estimated in exactly the same way as $I^+_{\lambda,  t}(x)$ above. It satisfies
\begin{equation}
\label{I-est3-2d}
| I^-_{\lambda, t}(x)|
\lesssim \lambda  ^{3/2}   t^{-1} .
\end{equation}

It remains to consider the stationary case, where $ |x |/t \sim   \lambda ^{-1/2}$. Note that $\partial_r \phi^-_{\lambda}(r)$ is strictly increasing for $r > 0$, since $\partial_r^2 \phi^-_{\lambda}(r) = - \lambda^2 m''(\lambda r)$ is strictly positive, by Lemma \ref{lm-mest}. Thus there is at most one point $r_0 \in [1/2,2]$ at which $\partial_r \phi^-_{\lambda}$ vanishes. Setting as before
\[
  \delta = t^{-1/2} \lambda^{-1/4},
\]
we limit our attention to the case where there is such a point $r_0$ in $[1/2+\delta,2-\delta]$; the remaining cases are treated by straightforward modifications of the following argument, much as in the 1d case in subsection \ref{1Dstationary}.

We decompose 
\begin{equation}
\label{I-decomp2d}
 I^-_{\lambda, t}(x)
=2\pi \lambda^2 \left( \int_{1/2}^{r_0-\delta} + \int_{{r_0-\delta}}^{{r_0+\delta}}+  \int_{{r_0+\delta}}^2 \right) e^{-it \phi^-_{\lambda} (r)  } \bar F_{\lambda}( |x|, r) \, dr.
\end{equation}
Integrating by parts we write the first integral as
\begin{align*}
& \int_{1/2}^{r_0-\delta} e^{-it \phi^-_{\lambda} (r)  } \bar F_{\lambda}( |x|, r)  \, dr
\\
&\qquad = \underbrace{i t^{-1}  \left[ e^{-it \phi^-_{\lambda} (r)  } \frac{ \bar F_{\lambda}( |x|, r) }{  \partial_r \phi^-_{\lambda}(r)}  \right ]_{r=1/2  }^{r_0-\delta}}_{=:A} 
  -\underbrace{ i t^{-1}\int_{1/2}^{r_0-\delta} e^{-it \phi^-_{\lambda} (r)  } \partial_r \left(
  \frac{ \bar F_{\lambda}( |x|, r) }{ \partial_r \phi^-_{\lambda}(r)} \right) \, dr}_{=:B}.
  \end{align*}
Using \eqref{Fest} and noting that for $r \in [1/2,r_0-\delta]$, $\partial_r \phi^-_{\lambda}(r)$ is negative and increasing, while $\partial_r^2 \phi^-_{\lambda}(r)$ is positive, we find
\[
|A|  \lesssim  t^{-1}(\lambda |x|)^{-1/2} \frac{1}{|\partial_r \phi^-_{\lambda}(r_0-\delta)|}
\]
and
\begin{align*}
|B| & \lesssim   t^{-1}  (\lambda |x|)^{-1/2} \left[ \int_{1/2}^{r_0-\delta} 
    \frac{1}{|\partial_r \phi^-_{\lambda}(r)  |} \, dr + \int_{1/2}^{r_0-\delta} 
     \frac{\partial_r^2 \phi^-_{\lambda}(r)}{[\partial_r \phi^-_{\lambda}(r)]^2} \, dr \right]
    \\
    &=   t^{-1}  (\lambda |x|)^{-1/2} \left[ \int_{1/2}^{r_0-\delta} 
    \frac{1}{|\partial_r \phi^-_{\lambda}(r)  |} \, dr + \int_{1/2}^{r_0-\delta}  \partial_r \left( \frac{1}{- \partial_r \phi^-_{\lambda}(r)} \right) \, dr \right]
    \\
    &\lesssim  t^{-1}(\lambda |x|)^{-1/2} \frac{1}{|\partial_r \phi^-_{\lambda}(r_0-\delta)|}
\end{align*}
But using \eqref{phi'+:est} and the mean value theorem, we see that
\[
 |\partial_r \phi^-_{\lambda}(r_0-\delta)| =  | \partial_r \phi^-_{\lambda}(r_0-\delta)-\partial_r \phi^-_{\lambda}(r_0)| 
 \sim  \lambda ^{1/2} \delta = \lambda^{1/4} t^{-1/2}.
\]
Using also the assumption $ |x |/t \sim   \lambda ^{-1/2}$, we conclude that
\begin{align*}
\left| \int_{1/2}^{r_0-\delta} e^{it \phi^-_{\lambda} (r)  } \bar F_{\lambda}( |x|, r)  \, dr\right| &\le |A|+ |B|
\\
&\lesssim  t^{-1} (\lambda |x|)^{-1/2}  \lambda^{-1/4}  t^{1/2}  
\\
&\sim  t^{-1} (\lambda^{1/2}t)^{-1/2}  \lambda^{-1/4}  t^{1/2}
\\
& =    t^{-1} \lambda^{-1/2}  .
\end{align*}
The third integral in \eqref{I-decomp2d} can also be estimated in a similar way, and satisfies the same estimate, while the second integral can be simply estimated as 
\begin{align*}
\left| \int_{r_0-\delta}^{r_0+\delta} e^{it \phi^-_{\lambda} (r)  }\bar F_{\lambda}( |x|, r) \, dr\right| \lesssim  (\lambda |x|)^{-1/2} \delta  \lesssim  t^{-1} \lambda^{-1/2} .
\end{align*}
Therefore, combining the above computations with \eqref{I-decomp2d} we have
\begin{equation}\label{I-est4-2d}
\left|I^-_{\lambda, t}(x) \right| \lesssim    \lambda ^{3/2} t^{-1},
\end{equation}
concluding the stationary case.

In summary, from \eqref{I-est2-2d}, \eqref{I-est3-2d} and \eqref{I-est4-2d} we obtain
\begin{equation*}
\begin{split}
|I_{\lambda, t}(x)| \le \sum_{\pm}|I^\pm_{\lambda, t}(x)| \lesssim  \lambda ^{3/2} t^{-1}
\end{split}
\end{equation*}
which is the desired estimate \eqref{Iest-1} with $d=2$.


\section{ Function spaces, Linear and bilinear estimates}
\label{linear_bilinear_estimates_section}
\setcounter{equation}{0}
\subsection{ Function spaces}
The mixed space-time Lebesgue space $L_t^q L_x^r$ on $\R^{d+1}$ is defined with the norm
$$
\|u\|_{L_t^q L_x^r}= \| \|u(t, \cdot)\|_{L_x^r}\|_{L_t^q }=\left(\int_{\R} \left(\int_{\R^d} |u(t,x)|^r \, dx\right)^{\frac qr} \, dt\right)^\frac 1q
$$
for $1\le q, r <\infty$ 
with an obvious modification when $q=\infty$ or $r=\infty$. 
We write $  L_T^q L_x^r$ when the time variable is restricted to the interval $[0, T]$.

Define the Bourgain space $X^{s,b}_\pm$ on $\R^{d+1}$ by the norm
\begin{align*}
\| u \|_{X^{s,b}_{\pm}} &= \left\| \angles{ \xi }^{s} \angles{ \tau \pm m_d(\xi)}^b \widetilde{u} (\xi, \tau) \right\|_{L^2_{\tau, \xi}},
\end{align*}
where $\widetilde u$ denotes the space-time Fourier transform given by
$$
\widetilde u(\tau, \xi)=\int_{\R^{d+1}} e^{ -i(t\tau+ x\cdot\xi)} u(t,x) \ dt dx.
$$
The restriction to the time slab $(0, T)\times \mathbb R^d$
of the Bourgain space, denoted by $X^{s, b}_\pm(T)$,
is a Banach space when equipped with the norm
$$
\| u \|_{X^{s, b}_\pm(T)}  = \inf \left\{ \| v \|_{X^{s, b}_\pm}: \ v = u \text{ on }  (0, T) \times \mathbb{R}^d \right\}.
$$

\subsection{ Linear estimates}
Let us recall some of the properties of these
spaces. We have
\begin{equation}\label{XH-C}
\sup_{0\le t\le T}\norm{u(t)}_{H^s}\le C \norm{u}_{X_\pm^{s,
b}(T)} \ \ \text{for} \ b>1/2.
\end{equation}
 For $-1/2<b'\le b<1/2$ and $0<T<1$ we have
 \begin{align}
 \label{DeltaFactor}
\norm{u}_{X_\pm^{s, b'}(T)} &\le C
T^{b-b'}\norm{u}_{X_\pm^{s, b}(T)},
\end{align}
where $C$ is independent on $T$.
The
proof for \eqref{XH-C} and \eqref{DeltaFactor} can
for instance be found in \cite{t06}. We recall also that for $2 \le q \le \infty$ and $b>1/2$,
\begin{equation}\label{timeSob}
  \norm{u}_{L^q_tL_x^2} \le C \norm{u}_{X^{0,b}_\pm},
\end{equation}
as can be seen by writing the left hand side as $\norm{e^{\pm i t m_d(D)} u}_{L^q_TL_x^2}$, applying Plancherel in $x$, then using Minkowski's integral inequality to switch the order of the norms to $L^2_\xi L^q_t$, and finally applying Sobolev embedding in $t$.

It is well known (see, e.g., \cite{dfs2007}) that
for any $s\in \R$ and $b>1/2$ one has
\begin{align}
\label{linXb}
\| S_{m_d}(\pm t) f \|_{X_\pm^{ s, b}(T)} &\leq C \| f \|_{H^{s}},
\\
\label{inhXb}
\left\| \int_0^t S_{m_d}(\pm( t - t')) F(t')\ dt' \right\|_{X_\pm^{s, b}(T)} &\leq C \| F \|_{X_\pm ^{ s, b - 1}(T)},
\end{align}
where the constant $C > 0$ depends only on $b$.

We need the following 
Bernstein inequality which is valid for $1 \le p \le  r \le \infty $ (see for instance \cite[Appendix A]{t06}):
\begin{equation}
\label{berstineq}
\norm{P_\lambda f}_{L^r(\R^d)} \le C \lambda^{\frac dp-\frac dr} \norm{P_\lambda f}_{L^p(\R^d)} 
\end{equation}
Another useful tool is the Hardy-Littlewood-Sobolev inequality (see \cite[Appendix A]{t06})
which asserts that\begin{equation}
\label{HLSineq}
\norm{|\cdot |^{-\alpha}\ast f}_{L^a(\R)} \le C \norm{ f}_{L^b(\R)} 
\end{equation}
whenever $1 < b < a < \infty$ and $0 < \alpha < 1$ obey the scaling condition
$$
\frac1b=\frac1a +1-\alpha.
$$

\begin{lemma}[Localized Strichartz estimates]\label{lm-LocStr}
Let $\lambda >1$ and $d \in \{1, 2\}$. Assume that $2 < q < \infty$ and $2 \le r \le \infty$ satisfy
\[
  \frac{2}{q} = \frac{d}{2} \left(1-\frac{2}{r}\right).
\]
Then we have the estimate 
\begin{equation}
\label{Str}
\norm{S_{m_d}(\pm t) f_\lambda }_{ L^q_{t} L^r_{x} (\R^{d+1}) }
\lesssim \lambda^{(3d/8)(1-2/r)} 
\norm{ f_\lambda }_{ L^2_{ x}(\R^d )}.
\end{equation}
Moreover, if $b>1/2$, we have
\begin{equation}
\label{Str1d-Transfer}
\norm{ u_\lambda }_{L^q_{t} L^r_{x}  (\R^{d+1}) } \lesssim 
\lambda^{(3d/8)(1-2/r)}
\norm{u_\lambda}_{  X^{0, b}_\pm }.
\end{equation}
\end{lemma}
\begin{proof}
By the standard $TT^*$-argument, \eqref{Str} is equivalent to the estimate 
\begin{equation}
\label{TTstar}
\norm{  K_{\lambda} \star F }_{L^{q}_{t} L^r_{ x} (\R^{d+1}) } \lesssim 
   \lambda^{(3d/4)(1-2/r)} 
\norm{ F  }_{ L^{q'}_{ t} L_x^{r'}(\R^{d+1} )},
\end{equation}
where $1/q+1/q'=1$ and $1/r+1/r'=1$, and where
\begin{equation}
\label{Km}
  K_{\lambda}(x,t)=\int_{\R^d} e^{i x\cdot\xi \pm itm_d (\xi)} \tilde \beta_{\lambda}(\xi)\, d\xi
\end{equation}
with $\tilde \beta_{\lambda}= \beta^2_{\lambda}$. Here $\star$ denotes the space-time convolution.   
Then by Corollary \ref{dispcor} (with $\beta$ replaced by $\beta^2$, which does not affect the validity of the corollary) we have the estimate
\[
\| K_{\lambda}(\cdot,t) * f  \|_{L^r_x(\R^d)} \lesssim  \left( \lambda^{3d/4} |t|^{-d/2} \right)^{1-2/r}  \|f\|_{L_x^{r'}(\R^d)}.
\]
Combining this with the Hardy-Littlewood-Sobolev inequality in the $t$
variable, with $(a, b)=(q , q')$ and $\alpha= (d/2)(1-2/r)$,
we estimate
\begin{align*}
  \norm{K_{\lambda}\star F }_{ L^q_t L^r_x  }
  &=
  \norm{\int_{\R} \int_{\R^d} K_{\lambda}(x-y,t-s) F(y,s) \, dy \, ds }_{ L^q_t L^r_x  }
  \\
  &\le \norm{   \int_{\R} \norm{ \int_{\R^d} K_{\lambda}(x-y,t-s)
   F(y,s) \, dy}_{L_x^r}  \, ds}_{L^q_t}
  \\
 &\lesssim
 \norm{ \int_{\R}  \left( \lambda^{3d/4} |t-s|^{-d/2} \right)^{1-2/r}
 \norm{ F(\cdot,s) }_{L_x^{r'} }  \, ds }_{L_t^q}
   \\
 &\lesssim \lambda^{(3d/4)(1-2/r)}  \norm{  
  \norm{ F }_{L_x^{r'} }  }_{L^{q'}_{ t} },
\end{align*}
proving \eqref{TTstar} and hence \eqref{Str}. By a standard argument, the latter implies \eqref{Str1d-Transfer} (see, for example, the proof of Lemma 2.9 in \cite{t06}).
\end{proof}

\subsection{Bilinear estimates}

Set $K=K(D) = K_d(D)$ for $d=1,2$, and note that the Fourier symbol equals in both dimensions
\[
  K(\xi) = \sqrt{\frac{\tanh\abs{\xi}}{\abs{\xi}}} \sim \angles{\xi}^{-1/2},
\]
hence
\begin{equation}\label{L2est}
 \|K f_\lambda\|_{L^2_x}   \lesssim \lambda^{-\frac12} \| f_\lambda\|_{L^2_x},
 \quad 
 \||D|K^2 f_\lambda\|_{L^2_x} \lesssim \| f_\lambda\|_{L^2_x} ,
 \quad
 \||D|K f_\lambda\|_{L^2_x} \lesssim \lambda^\frac12 \| f_\lambda\|_{L^2_x} .
\end{equation}

We first note the following consequence of the localized Strichartz estimates in Lemma \ref{lm-LocStr}.

\begin{lemma}
\label{L2lemma}
Let $b > 1/2$ and dyadic $\lambda_1, \lambda_2 \geq 1$.
For $d=1, 2$ and $1 < p \leqslant 2 < r \leqslant \infty$
satisfying
\[
	\frac 12 + \frac 1r = \frac 1p
\]
we have the estimate
\[
\norm{u_{\lambda_1}  v_{\lambda_2}}_{L_t^2L^p_x}  
\lesssim \min(\lambda_1,\lambda_2)^{(3d/8) (1 - 2/r)} 
\norm{u_{\lambda_1}}_{ X_\pm^{0, b} } 
\norm{v_{\lambda_2} }_{X_\pm^{0, b}}
\]
provided $p < 2$ in the case $d = 2$.

For $d=2$ and $2 < r < \infty$, we have for all $T > 0$,
\[
\norm{u_{\lambda_1}  v_{\lambda_2}}_{L_T^2 L^2_x}  
\lesssim  T^{1/r} \min(\lambda_1,\lambda_2)^{3/4+1/(2r)}
\norm{u_{\lambda_1}}_{ X_\pm^{0, b} } 
\norm{v_{\lambda_2} }_{X_\pm^{0, b}}.
\]
In both estimates,
the signs in the $X_\pm$ norms can be chosen
independently on each other.
\end{lemma}

\begin{proof}
By symmetry we may assume 
$1\le \lambda_1 \le \lambda_2$.
Consider first the case $d=1$. By H\"{o}lder's inequality
and \eqref{timeSob},
\[
  \norm{u_{\lambda_1}v_{\lambda_2}}_{L_t^2L^p_x} \le
  \norm{u_{\lambda_1}}_{ L_t^q L_x^r  } 
  \norm{ v_{\lambda_2} }_{ L_t^{q_1} L_x^2   }
  \lesssim
  \norm{u_{\lambda_1}}_{ L_t^q L_x^r  } 
  \norm{ v_{\lambda_2} }_{X_\pm^{0, b}},
\]
where $q$ is taken as in Lemma \ref{lm-LocStr}
and
$1/q + 1/q_1 = 1/2$.
So it only remains to check that
\[
	\norm{u_{\lambda_1}}_{ L_t^q L_x^r  }
	\lesssim
	\lambda_1^{3/8 (1 - 2/r)} \norm{u_{\lambda_1}}_{ X_\pm^{0, b} },
\]
but this holds by Lemma \ref{lm-LocStr} if $\lambda_1 > 1$,
while if $\lambda_1=1$ we can use the Bernstein inequality 
\eqref{berstineq} followed by \eqref{timeSob} to obtain
\[
	\norm{u_{\lambda_1}}_{ L_t^q L_x^r  }
	\lesssim \norm{u_{\lambda_1}}_{ L_t^q L_x^2  }
	\lesssim
	\norm{u_{\lambda_1}}_{  X_\pm^{0, b} } .
\]
Similarly, one obtains the first estimate for $p < d = 2$.

Now consider the case $p = d = 2$.
We apply H\"{o}lder's inequality and \eqref{timeSob} to write
\[
  \norm{u_{\lambda_1}v_{\lambda_2}}_{L_T^2L^2_x} \le
  \norm{u_{\lambda_1}}_{ L_T^2 L_x^\infty  } 
  \norm{ v_{\lambda_2} }_{ L_T^\infty L_x^2  }
  \lesssim 
  \norm{u_{\lambda_1}}_{ L_T^2 L_x^\infty  } 
  \norm{ v_{\lambda_2} }_{X_\pm^{0, b}}.
\]
To estimate $\norm{u_{\lambda_1}}_{ L_T^2 L_x^\infty  }$ we want to use Lemma \ref{lm-LocStr}, so we let $2 < r < \infty$ and define $q$ by $2/q=1-2/r$. Thus $1/2=1/q+1/r$, so  applying H\"older in $t$,  the Bernstein inequality in $x$, and finally Lemma \ref{lm-LocStr}, we get
\[
  \norm{u_{\lambda_1}}_{ L_T^2 L_x^\infty  }
  \le
  T^{1/r} \lambda_1^{2/r} \norm{u_{\lambda_1}}_{ L_T^q L_x^r  }
  \lesssim
  T^{1/r} \lambda_1^{2/r}  \lambda_1^{(3/4)(1-2/r)} \norm{u_{\lambda_1}}_{ X_\pm^{0, b} },
\]
proving the claimed estimate in the case $\lambda_1 > 1$. If $\lambda_1 = 1$, we can apply the Bernstein inequality and \eqref{timeSob}, instead of Lemma \ref{lm-LocStr}, and again we get the desired estimate.
\end{proof}

We now present the key bilinear space-time estimates needed for the proof of local well-posedness. 

\begin{lemma}
\label{lm-bilest}
Let $1/2<b<1$ and $0 < T < 1$. Assume that $s_d>-1/10 $  if $d=1$ and $s_d>1/4$ if $d=2$.
Then we have the estimates
\begin{align}
\label{biest10}
\norm{  |D| K^2 \left( u\cdot  Kv \right) }_{X_\pm^{s_d, b-1}(T)}  
\lesssim   T^{1-b}
\norm{u}_{ X_\pm^{s_d, b} } 
\norm{v}_{X_\pm^{s_d, b}} ,
\\
\label{biest21}
\norm{  |D|K \left(  Ku \cdot Kv \right) }_{X_\pm^{s_d, b-1}(T)}  
\lesssim  T^{1-b}
\norm{u}_{ X_\pm^{s_d, b} } 
\norm{v }_{X_\pm^{s_d, b}} ,
\end{align}
where the signs in all the $X_\pm$ norms can be chosen independently on each other.
\end{lemma}

\begin{proof}[Proof of \eqref{biest10}]
In view of \eqref{DeltaFactor} the estimate \eqref{biest10} reduces to proving
\[
\norm{    |D| K^2 \left( u\cdot  Kv \right) }_{L_T^2 H_x^{s_d}}  
\lesssim
\norm{u}_{ X_\pm^{s_d, b} } 
\norm{v}_{X_\pm^{s_d, b}},
\]
which by duality can be reduced to
\begin{equation}\label{duality-biest11}
\left| \int_0^T \int_{\R^d}   |D| K^2\angles{D}^{s_d}\left( \angles{D}^{-s_d} u\cdot  \angles{D}^{-s_d}Kv \right) w \ dx dt \right| \lesssim
\norm{u}_{  X^{0, b}_\pm} \norm{v}_{  X^{0, b}_\pm } \norm{ w }_{L^2_{t,x}}.
\end{equation}
Decomposing $u= \sum_{\lambda_1 \ge 1} u_{\lambda_1}$ and $v= \sum_{\lambda_2 \ge 1} v_{\lambda_2}$ 
we have
\begin{equation}\label{duality-biestdecomp}
\text{LHS \eqref{duality-biest11} } \lesssim \sum_{\lambda, \lambda_1, \lambda_2 \ge 1} \left| \int_0^T \int_{\R^d}  |D| K^2\angles{D}^{s_d} P_\lambda \left( \angles{D}^{-{s_d}} u_{\lambda_1}\cdot  \angles{D}^{-{s_d}}  K v_{\lambda_2} \right) w_\lambda \ dx dt \right|.
\end{equation}
Setting
$$
a_{\lambda_1}:= \norm{u_{\lambda_1}}_{  X^{0, b}_\pm }, \quad b_{\lambda_2}:= \norm{v_{\lambda_2}}_{  X^{0, b}_\pm }, \quad c_{\lambda}:= \norm{w_{\lambda}}_{L^2_{t,x}   }
$$
we have
$$
\norm{u}_{  X^{0, b}_\pm }\sim \|(a_{\lambda_1})\|_{l^2_{\lambda_1}}, \quad  
\norm{v}_{  X^{0, b}_\pm }\sim \|(b_{\lambda_2})\|_{l^2_{\lambda_2}} , \quad   
\norm{w}_{  L^2_{t,x} }\sim \|(c_{\lambda})\|_{l^2_{\lambda}},
$$
hence the estimate \eqref{duality-biest11} reduces to proving 
\begin{equation}\label{duality-biest11-r}
\text{RHS \eqref{duality-biestdecomp}}\lesssim \|(a_{\lambda_1})\|_{l^2_{\lambda_1}} \|(b_{\lambda_2})\|_{l^2_{\lambda_2}} \|(c_{\lambda})\|_{l^2_{\lambda}}.
\end{equation}
To this end, we note that by Lemma \ref{L2lemma} we have, for $\varepsilon > 0$ arbitrarily small
and $1 < p \leqslant 2$ with $1/r = 1/p - 1/2$,
\begin{equation}\label{L2bilinear}
	\norm{
		P_\lambda
			\left( u_{\lambda_1}  v_{\lambda_2} \right)
	}_{L_T^2L^p_x}  
	\lesssim
	\min(\lambda_1,\lambda_2)^{3d/8 (1 - 2/r) +\varepsilon}
	\norm{u_{\lambda_1}}_{ X_\pm^{0, b} } 
  	\norm{v_{\lambda_2} }_{X_\pm^{0, b}}.
\end{equation}
We remark that in dimension $d=1$,
the lemma would actually allow us to take $\varepsilon = 0$,
but the proof below works for sufficiently small,
positive $\varepsilon > 0$ in both dimensions.
Note apart from Estimate \eqref{L2bilinear}
that all the non-trivial terms of RHS \eqref{duality-biestdecomp}
can be separated into the three groups
\(
	I_1, I_2, I_3
\)
with
\(
	\lambda \lesssim \lambda_1 \sim \lambda_2
	,
	\lambda_1 \ll \lambda_2 \sim \lambda
\)
and
\(
	\lambda_2 \ll \lambda_1 \sim  \lambda
	,
\)
respectively.
It is a straightforward consequence of taking $\lambda$-projection
of the product of $\lambda_1, \lambda_2$-projections.
For each of these groups we can make a particular choice
of $p$ applying Estimate \eqref{L2bilinear}.

Thus using Cauchy-Schwarz,
the Bernstein inequality \eqref{berstineq},
\eqref{L2est} and \eqref{L2bilinear} we obtain
\begin{equation}\label{maindecomp}
\begin{split}
\text{RHS \eqref{duality-biestdecomp}}
&\lesssim
 \sum_{\lambda, \lambda_1, \lambda_2 \ge 1}
\norm{  |D| K^2\angles{D}^{s_d} P_\lambda \left( \angles{D}^{-s_d} u_{\lambda_1}\cdot  \angles{D}^{-s_d} K v_{\lambda_2} \right) }_{L_{T,x}^2} \norm{w_{\lambda} }_{ L_{T,x}^2  }
\\
&\lesssim
\sum_{
	\substack{ \lambda, \lambda_1, \lambda_2 \ge 1
	\\
	P_{\lambda}(\ldots) \neq 0 }
}
\lambda^{s_d + d/p - d/2}
\min(\lambda_1,\lambda_2)^{3d/8 (1 - 2/r) +\varepsilon}
 \lambda_1^{-s_d} \lambda_2^{-1/2-s_d }  
a_{\lambda_1} b_{\lambda_2} c_{\lambda}
\\
&\lesssim  I_1(d) + I_2(d) + I_3(d),
\end{split}
\end{equation}
where
\begin{align*}
I_1(d)&= 
 \sum_{  \substack{ \lambda, \lambda_1, \lambda_2 \ge 1\\ \lambda\lesssim  \lambda_1\sim  \lambda_2 } }
\lambda^{s_d + d/p - d/2}
\lambda_2^{3d/8 (1 - 2/r) + \varepsilon -1/2-2s_d } 
a_{\lambda_1} b_{\lambda_2} c_{\lambda}, 
\\
\quad I_2(d)&= \sum_{  \substack{ \lambda, \lambda, \lambda_2 \ge 1\\  \lambda_1 \ll \lambda_2\sim  \lambda } }  \left(\frac{\lambda_1}{\lambda_2}\right)^{1/2}  \lambda_1^{3d/8+\varepsilon-1/2-s_d } 
a_{\lambda_1} b_{\lambda_2} c_{\lambda}, 
\\
\quad I_3(d)&=\sum_{  \substack{\lambda,  \lambda_1, \lambda \ge 1\\    \lambda_2 \ll \lambda_1\sim  \lambda} }  \lambda_2^{3d/8+\varepsilon-1/2-s_d } 
a_{\lambda_1} b_{\lambda_2} c_{\lambda}.
\end{align*}

We first estimate $I_1(1)$.
Notice
\(
	\lambda^{s_1 + 1/p - 1/2}
	\lambda_2^{3/8 (1 - 2/r) + \varepsilon - 1/2 - 2s_1 } 
	\lesssim
	\lambda^{1/p - 1 + 3/8 (1 - 2/r) + \varepsilon - s_1 }
\)
provided
\(
	3/8 (1 - 2/r) + \varepsilon - 1/2 - 2s_1 < 0
\)
or equivalently if
\(
	s_1 > - (1 + 6/r)/16 + \varepsilon / 2
	.
\)
Consequently, we can apply the Cauchy-Schwarz inequality first in $\lambda_1 \sim \lambda_2$ and then in $\lambda$ to estimate $I_1(1)$ as
\[
	I_1(1)
	\lesssim
	\left(
		\sum_{\lambda \ge 1}
		\lambda^{ (2/r - 1)/8 + \varepsilon - s_1 }
		c_{\lambda} 
	\right)
	\|(a_{\lambda_1})\|_{l^2_{\lambda_1}} \|(b_{\lambda_2})\|
	_{l^2_{\lambda_2}} \lesssim
	\|(a_{\lambda_1})\|_{l^2_{\lambda_1}} \|(b_{\lambda_2})\|
	_{l^2_{\lambda_2}} \|(c_{\lambda})\|_{l^2_{\lambda}}
\]
provided
\(
	s_1 > (2/r - 1)/8 + \varepsilon
\)
and
\(
	s_1 > - (1 + 6/r)/16 + \varepsilon / 2
	.
\)
The lowest possible bound $s_1 > -1/10$ is obtained by taking $r = 10$,
since $\varepsilon > 0$ is arbitrary small
(and can be taken actually zero according to Lemma \ref{L2lemma}).
This corresponds to $p = 5/3$.
To estimate the other sums it is enough to stick to $p = 2$
everywhere below.

Next we estimate $I_1(2)$ with $p = 2$.
Since $s_2>1/4$ by assumption, we have $\lambda^{s_2} \lambda_2^{\varepsilon+1/4-2s_2 } \lesssim ( \lambda/\lambda_2)^{s_2 }$.
Then we apply the Cauchy-Schwarz inequality first in $\lambda$ and then in $\lambda_1 \sim \lambda_2$ to obtain the desired estimate:
\[
I_1(2) \lesssim \sum_{  \lambda_1\sim  \lambda_2 }  \left(
\sum_{\lambda\lesssim \lambda_2}  ( \lambda/\lambda_2)^{s_2 } c_{\lambda} \right) 
a_{\lambda_1} b_{\lambda_2}
\lesssim
\|(a_{\lambda_1})\|_{l^2_{\lambda_1}} \|(b_{\lambda_2})\|_{l^2_{\lambda_2}} \|(c_{\lambda})\|_{l^2_{\lambda}}.
\]

Next we estimate $I_3(d)$. Since $s_d>3d/8-1/2$, we have $\varepsilon+3d/8-1/2-s_d < 0$, for $\varepsilon > 0$ small enough. Applying the Cauchy-Schwarz inequality first in $\lambda_1 \sim \lambda$ and then in $\lambda_2$, we get
\[
  I_3(d)
  \lesssim
  \left( \sum_{\lambda_2 \ge 1}  \lambda_2^{3d/8+\varepsilon-1/2-s_d } b_{\lambda_2} \right)
  \|(a_{\lambda_1})\|_{l^2_{\lambda_1}} \|(c_{\lambda})\|_{l^2_{\lambda}}
  \lesssim
  \|(a_{\lambda_1})\|_{l^2_{\lambda_1}} \|(b_{\lambda_2})\|_{l^2_{\lambda_2}} \|(c_{\lambda})\|_{l^2_{\lambda}}.
\]

Finally, we note that in $I_2(d)$, we can discard the small factor $(\lambda_1/\lambda_2)^{1/2}$ and reduce to the same estimate as for $I_3(d)$. This completes the proof of \eqref{biest10}.
\end{proof}

\begin{proof}[Proof of \eqref{biest21}]
We follow the same argument as in the proof of \eqref{biest10}. By duality and dyadic decomposition, \eqref{biest21} reduces to proving
\[
  S \lesssim \|(a_{\lambda_1})\|_{l^2_{\lambda_1}} \|(b_{\lambda_2})\|_{l^2_{\lambda_2}} \|(c_{\lambda})\|_{l^2_{\lambda}},
\]
where
\[
  S = \sum_{\lambda, \lambda_1, \lambda_2 \ge 1} \left|\int_0^T \int_{\R^d} |D| K\angles{D}^{s_d} P_\lambda \left(  \angles{D}^{-s_d} Ku_{\lambda_1}\cdot \angles{D}^{-s_d} K v_{\lambda_2} \right) w_\lambda \ dx dt \right| .
\]
By Cauchy-Schwarz, \eqref{L2est} and \eqref{L2bilinear} we obtain
\begin{align*}
  S
  &\lesssim
\sum_{
	\substack{ \lambda, \lambda_1, \lambda_2 \ge 1
	\\
	P_{\lambda}(\ldots) \neq 0 }
}
  \lambda^{\frac12+s_d}  \min(\lambda_1,\lambda_2)^{3d/8+\varepsilon}
  (\lambda_1\lambda_2)^{-1/2 -s_d}
  a_{\lambda_1} b_{\lambda_2} c_{\lambda}
  \\
  &= 
\sum_{
	\substack{ \lambda, \lambda_1, \lambda_2 \ge 1
	\\
	P_{\lambda}(\ldots) \neq 0 }
}
  \left( \frac{\lambda}{\lambda_1} \right)^{1/2}
  \lambda^{s_d}  \min(\lambda_1,\lambda_2)^{3d/8+\varepsilon} \lambda_1^{-s_d} \lambda_2^{-1/2-s_d}
a_{\lambda_1} b_{\lambda_2} c_{\lambda},
\end{align*}
and comparing with the corresponding sum \eqref{maindecomp} from the proof of \eqref{biest10}, we see that the only difference is that we now have an extra factor $(\lambda/\lambda_1)^{1/2}$.
This factor is bounded except for the case
$\lambda_1 \ll \lambda_2 \sim \lambda$,
so it is enough to consider $I_2(d)$ with this factor inserted:
\[
  I_2'(d) = \sum_{  \substack{ \lambda, \lambda, \lambda_2 \ge 1\\  \lambda_1 \ll \lambda_2\sim  \lambda } }  \left( \frac{\lambda}{\lambda_1} \right)^{1/2} \left(\frac{\lambda_1}{\lambda_2}\right)^{1/2}  \lambda_1^{3d/8+\varepsilon-1/2-s_d } 
a_{\lambda_1} b_{\lambda_2} c_{\lambda}
\lesssim
\sum_{  \substack{ \lambda, \lambda, \lambda_2 \ge 1\\  \lambda_1 \ll \lambda_2\sim  \lambda } }  	  \lambda_1^{3d/8+\varepsilon-1/2-s_d } 
a_{\lambda_1} b_{\lambda_2} c_{\lambda}.
\]
But the right hand side was already estimated in the proof of \eqref{biest10} (the estimate for $I_3(d)$). This completes the proof of \eqref{biest21}.
\end{proof}

\section{Proof of Theorem \ref{lwpthm}, Theorem \ref{lwpthm2d-1} and Theorem \ref{mainthm}}
\setcounter{equation}{0}

\subsection{Proof of Theorems \ref{lwpthm} and Theorem \ref{lwpthm2d-1} }
We solve the integral equations \eqref{inteqwt1} and \eqref{inteqwt2d-1} by contraction mapping techniques
as follows. Define the mapping
\[
  (u^+_d, u^-_d) \mapsto (\Phi_+(u^+_d, u^-_d),\Phi_-(u^+_d, u^-_d))
\]
by
\[
\Phi_\pm(u^+_d, u^-_d)(t) :=S_{m_d}(\pm t) f^\pm_d -i  \int_0^t S_{m_d}( \pm (t-s) ) B^\pm_d(u^+_d, u^-_d)(s) \, ds.
\]
Let 
\[
  R_d = \|f^+_d\|_{H^{s_d}} +\|f^-_d\|_{H^{s_d}}.
\]
We look for a solution in the set
\[
  \mathcal D(R_d) = \left\{  (u^+_d, u^-_d) \in X^{s_d,b}_+(T) \times X^{s_d,b}_-(T)
  \colon \| u^+_d\|_{ X^{s_d,b}_+(T)} + \| u^-_d\|_{ X^{s_d,b}_-(T)} \le 4 C R_d \right\}
\]
where $b \in (1/2, 1)$ and
$C$ is as in \eqref{linXb}, \eqref{inhXb}. Now for $(u^+_d, u^-_d) \in \mathcal D(R_d)$ we have by \eqref{linXb}, \eqref{inhXb} and Lemma \ref{lm-bilest},
$$
  \| \Phi_+(u^+_d, u^-_d)\|_{ X^{s_d,b}_+(T)} + \| \Phi_-(u^+_d, u^-_d)\|_{ X^{s_d,b}_-(T)} \le 2C  R_d + C'  T^{1-b} R_d^2 \le 4CR_d,
$$
where the last inequality certainly holds provided that
\[
  T =  \left(\frac{1}{16 CC' (1+R_d)} \right)^{\frac{1}{1-b}}.
\]
Moreover, for $(u^+_d, u^-_d)$ and $(v^+_d, v^-_d)$ in $\mathcal D(R_d)$ with the same data, one can show similarly the difference estimate
\begin{align*}
 &\sum_\pm \| \Phi_\pm(u^+_d, u^-_d)-\Phi_\pm(v^+_d, v^-_d)\|_{ X^{s_d,b}_\pm(T)}
 \\
 & \quad \le C' T^{1-b} \left( \sum_{\pm} \| u^\pm_d-v^\pm_d\|_{ X^{s_d,b}_\pm(T)} \right)
 \left( \sum_{\pm} \left( \| u^\pm_d\|_{ X^{s_d,b}_\pm(T)}+  \| v^\pm_d\|_{ X^{s_d,b}_\pm(T)} \right) \right)
 \\
 &\quad \le 8CC' R_d T^{1-b} \left( \sum_{\pm} \| u^\pm_d-v^\pm_d\|_{ X^{s_d,b}_\pm(T)} \right).
\end{align*}
With $T$ chosen as above, the constant $8CC' R_d T^{1-b}$ is strictly less than one, hence $(\Phi_+,\Phi_-)$ is a contraction on $\mathcal D(R_d)$ and therefore it has a unique fixed point $(u^+_d, u^-_d) \in \mathcal D(R_d)$ solving the integral equation on $\R^d\times [0, T]$. Uniqueness in the whole space $X^{s_d,b}_+(T) \times X^{s_d,b}_-(T)$ and continuous dependence on the initial data can be shown in a similar way, by the difference estimates. This concludes the proof of Theorems \ref{lwpthm} and \ref{lwpthm2d-1}.

Then we use the transformation \eqref{upm} to obtain the solution 
$$(\eta,  v) \in C\left( [0, T]; H^{s_1}(\R) \times  H^{s_1+1/2}(\R)\right)$$ of the original system
\eqref{wt1}--\eqref{data1}. Similarly, we use the transformation \eqref{upm2d} 
to obtain the solution 
$$(\eta,  \mathbf v) \in C\left( [0, T]; H^{s_2}\left( \R^2 \right) \times  
\left( H^{{s_2}+1/2} \left( \R^2 \right) \right)^2 \right)$$ of the original system
\eqref{wt2d}--\eqref{data2d}. Thus we obtain also Theorems \ref{simple_theorem} and \ref{simple_theorem2d}.

\subsection{Proof of Theorem \ref{mainthm}}

Here we assume $d=1$. For $s = 0$ one can easily extend the local result globally
making use of Lemma \ref{global_lemma_s_0}.
With the global bound of the lemma we can reapply
the local result, Theorem \ref{simple_theorem}, as many times
as we want, thus proving Theorem \ref{mainthm}
with $\delta = \epsilon_0 / 2$ for $s = 0$.
The proof for positive $s$ is done iteratively.
In other words, assuming the result for some $s' \geq 0$
we prove for $s \in (s', s' + 1/4]$.
The argument is essentially the persistence of regularity
based on the a priori estimate lemma \ref{global_lemma_s_positive},
where we use the notation
\(
	\lVert (\eta, v) \rVert _{X^s}
\)
defined by \eqref{Xs_norm_definition}.
Indeed, the first estimate in Lemma \ref{global_lemma_s_positive}
allows to reapply the
local result and extend the solution to
any time interval if $0 < s < 1/2$.
In the case $s \geqslant 1/2$ extension is carried out
iteratively making use of
the second inequality in Lemma \ref{global_lemma_s_positive}.


\vskip 0.05in
\noindent
{\bf Acknowledgments.}
{
The authors are grateful to Didier Pilod for
fruitful discussions and numerous helpful comments
on a preliminary version of this work.
They also thank the anonymous referee for helpful comments and suggestions.
We acknowledge support from the
Norwegian Research Council and the Trond Mohn Stiftelse (project \emph{Pure Mathematics in Norway}).
}

\bibliographystyle{amsplain}
\bibliography{bibliography}

\providecommand{\bysame}{\leavevmode\hbox to3em{\hrulefill}\thinspace}
\providecommand{\MR}{\relax\ifhmode\unskip\space\fi MR }
\providecommand{\MRhref}[2]{%
  \href{http://www.ams.org/mathscinet-getitem?mr=#1}{#2}
}
\providecommand{\href}[2]{#2}
\begin{thebibliography}{10}

\bibitem{Ai2019}
{Ai Albert}, \emph{{Low Regularity Solutions for Gravity Water Waves}}, Water
  Waves \textbf{1} (2019), no.~1, 145--215.

\bibitem{Alazard_Burq_Zuily2018}
T.~Alazard, N.~Burq, and C.~Zuily, \emph{{Strichartz Estimates and the Cauchy
  Problem for the Gravity Water Waves Equations}}, Memoirs of the American
  Mathematical Society \textbf{256} (2018), no.~1229. \MR{3852259}

\bibitem{Benoit}
M{\'e}sognon-Gireau Beno\^{\i}t, \emph{{A dispersive estimate for the
  linearized water-waves equations in finite depth}}, J. Math. Fluid Mech.
  \textbf{19} (2017), no.~3, 469--500. \MR{3685970}

\bibitem{Bona_Tzvetkov}
Jerry~L. Bona and Nikolay Tzvetkov, \emph{{Sharp well-posedness results for the
  {BBM} equation}}, Discrete Contin. Dyn. Syst. \textbf{23} (2009), no.~4,
  1241--1252. \MR{2461849}

\bibitem{Brezis_Gallouet}
H.~Br{\'e}zis and T.~Gallouet, \emph{{Nonlinear {S}chr{\"o}dinger evolution
  equations}}, Nonlinear Anal. \textbf{4} (1980), no.~4, 677--681. \MR{582536}

\bibitem{Brezis_Wainger}
Ha\"{\i}m Br{\'e}zis and Stephen Wainger, \emph{{A note on limiting cases of
  {S}obolev embeddings and convolution inequalities}}, Comm. Partial
  Differential Equations \textbf{5} (1980), no.~7, 773--789. \MR{579997}

\bibitem{Carter}
John~D. Carter, \emph{{Bidirectional {W}hitham equations as models of waves on
  shallow water}}, Wave Motion \textbf{82} (2018), 51--61. \MR{3844340}

\bibitem{dfs2007}
Piero D'Ancona, Damiano Foschi, and Sigmund Selberg, \emph{{Null structure and
  almost optimal local regularity for the {D}irac-{K}lein-{G}ordon system}}, J.
  Eur. Math. Soc. (JEMS) \textbf{9} (2007), no.~4, 877--899. \MR{2341835}

\bibitem{Dinvay}
Evgueni Dinvay, \emph{{On well-posedness of a dispersive system of the
  {W}hitham-{B}oussinesq type}}, Appl. Math. Lett. \textbf{88} (2019), 13--20.
  \MR{3862707}

\bibitem{Dinvay_Dutykh_Kalisch}
Evgueni Dinvay, Denys Dutykh, and Henrik Kalisch, \emph{{A comparative study of
  bi-directional {W}hitham systems}}, Appl. Numer. Math. \textbf{141} (2019),
  248--262. \MR{3944703}

\bibitem{Duchene_Israwi}
V.~Duch{\^e}ne, S.~Israwi, and R.~Talhouk, \emph{{A new class of two-layer
  {G}reen-{N}aghdi systems with improved frequency dispersion}}, Stud. Appl.
  Math. \textbf{137} (2016), no.~3, 356--415. \MR{3564304}

\bibitem{Duchene_Nilsson_Wahlen}
Vincent Duch{\^e}ne, Dag Nilsson, and Erik Wahl{\'e}n, \emph{{Solitary wave
  solutions to a class of modified {G}reen-{N}aghdi systems}}, J. Math. Fluid
  Mech. \textbf{20} (2018), no.~3, 1059--1091. \MR{3841973}

\bibitem{G08}
Loukas Grafakos, \emph{{Classical {F}ourier analysis}}, third ed., {Graduate
  Texts in Mathematics}, vol. 249, Springer, New York, 2014. \MR{3243734}

\bibitem{Hormander}
Lars H{\"o}rmander, \emph{{Lectures on nonlinear hyperbolic differential
  equations}}, {Math{\'e}matiques \& Applications (Berlin) [Mathematics \&
  Applications]}, vol.~26, Springer-Verlag, Berlin, 1997. \MR{1466700}

\bibitem{Hur_Pandey}
Vera~Mikyoung Hur and Ashish~Kumar Pandey, \emph{{Modulational instability in a
  full-dispersion shallow water model}}, Stud. Appl. Math. \textbf{142} (2019),
  no.~1, 3--47. \MR{3897262}

\bibitem{Hur_Tao}
Vera~Mikyoung Hur and Lizheng Tao, \emph{{Wave breaking in a shallow water
  model}}, SIAM J. Math. Anal. \textbf{50} (2018), no.~1, 354--380.
  \MR{3749383}

\bibitem{Kalisch_Pilod}
Henrik Kalisch and Didier Pilod, \emph{{On the local well-posedness for a
  full-dispersion {B}oussinesq system with surface tension}}, Proc. Amer. Math.
  Soc. \textbf{147} (2019), no.~6, 2545--2559. \MR{3951431}

\bibitem{Klein_Linares_Pilod}
Christian Klein, Felipe Linares, Didier Pilod, and Jean-Claude Saut, \emph{{On
  {W}hitham and related equations}}, Stud. Appl. Math. \textbf{140} (2018),
  no.~2, 133--177. \MR{3763731}

\bibitem{Lannes}
David Lannes, \emph{{The water waves problem}}, {Mathematical Surveys and
  Monographs}, vol. 188, American Mathematical Society, Providence, RI, 2013,
  Mathematical analysis and asymptotics. \MR{3060183}

\bibitem{Linares_Ponce}
Felipe Linares and Gustavo Ponce, \emph{{Introduction to nonlinear dispersive
  equations}}, second ed., {Universitext}, Springer, New York, 2015.
  \MR{3308874}

\bibitem{Nilsson_Wang}
Dag Nilsson and Yuexun Wang, \emph{{Solitary wave solutions to a class of
  {W}hitham-{B}oussinesq systems}}, Z. Angew. Math. Phys. \textbf{70} (2019),
  no.~3, Art. 70, 13. \MR{3936049}

\bibitem{Pei_Wang}
Long Pei and Yuexun Wang, \emph{{A note on well-posedness of bidirectional
  Whitham equation}}, Applied Mathematics Letters \textbf{98} (2019), 215--223.

\bibitem{Ponce}
Gustavo Ponce, \emph{{On the global well-posedness of the {B}enjamin-{O}no
  equation}}, Differential Integral Equations \textbf{4} (1991), no.~3,
  527--542. \MR{1097916}

\bibitem{S93}
Elias~M. Stein, \emph{{Harmonic analysis: real-variable methods, orthogonality,
  and oscillatory integrals}}, {Princeton Mathematical Series}, vol.~43,
  Princeton University Press, Princeton, NJ, 1993, With the assistance of
  Timothy S. Murphy, Monographs in Harmonic Analysis, III. \MR{1232192}

\bibitem{t06}
Terence Tao, \emph{{Nonlinear dispersive equations}}, {CBMS Regional Conference
  Series in Mathematics}, vol. 106, Published for the Conference Board of the
  Mathematical Sciences, Washington, DC; by the American Mathematical Society,
  Providence, RI, 2006, Local and global analysis. \MR{2233925}

\end{thebibliography}

\end{document}